\documentclass[11pt]{amsart}

\usepackage{amsmath,amsthm,amssymb,amsfonts}
\usepackage{times,microtype}
\usepackage{mathtools}
\usepackage{stmaryrd}
\usepackage{cases}
\usepackage{empheq}
\usepackage{paralist}
\usepackage{hyperref} 
\usepackage{multimedia}
\usepackage{tikz}
\usepackage{enumitem}
\usepackage{cite}
\usepackage{subfig}
\usepackage{multirow}

\usepackage[margin=1.1in]{geometry}

\usepackage[explicit]{titlesec}
\titleformat{\section}{\large\bfseries\center\raggedright}{\thesection}{0.5em}{{#1}}[]
\titleformat{\subsection}{\bfseries\center\raggedright}{\thesubsection}{0.5em}{{#1}}[]
\titleformat{\subsubsection}[runin]{\bfseries}{\thesubsubsection}{0.5em}{{#1}}[.]
\titlespacing*{\section}{0pt}{0.8\baselineskip}{0.6\baselineskip}
\titlespacing*{\subsection}{0pt}{0.6\baselineskip}{0.4\baselineskip}
\titlespacing*{\subsubsection}{0pt}{0.4\baselineskip}{0.4\baselineskip}

\setenumerate{itemsep=-0.2em,topsep=0.3em}
\setdescription{itemsep=-0.2em,topsep=0.3em}
\setitemize{itemsep=-0.2em,topsep=0.3em}

\newcommand\namefont{\normalfont\bfseries}
\newcommand\numberfont{\normalfont\bfseries}
\newcommand\notefont{\normalfont\bfseries}
\newtheoremstyle{mystyle} 
	{0.3em} 
	{0.3em} 
	{\itshape} 
	{} 
	{\normalfont} 
	{.} 
	{.5em} 
	{{\namefont\thmname{#1}}~{\numberfont\thmnumber{#2}}{\notefont\thmnote{ (#3)}}} 
\theoremstyle{mystyle}
\newtheorem{thm}{Theorem}[section]

\newtheorem{lem}[thm]{Lemma}

\newtheorem{prop}[thm]{Proposition}
\newtheorem{hyp}[thm]{Assumption}

\newtheorem{defn}[thm]{Definition}

\newtheoremstyle{mystyle2} 
	{0.3em} 
	{0.3em} 
	{\normalfont} 
	{} 
	{\normalfont} 
	{.} 
	{.5em} 
	{{\namefont\thmname{#1}}~{\numberfont\thmnumber{#2}}{\notefont\thmnote{ (#3)}}} 
\theoremstyle{mystyle2}
\newtheorem{rem}[thm]{Remark}

\allowdisplaybreaks
\mathtoolsset{showonlyrefs}
\renewcommand{\leq}{\leqslant}
\renewcommand{\geq}{\geqslant}

\newcommand{\te}{\textrm}

\definecolor{dred}{rgb}{0.75,0,0}
\definecolor{darkgreen}{rgb}{0.1,0.6,0.1}
\definecolor{darkblue}{rgb}{0.1,0.1,0.6}

\definecolor{mygreen}{rgb}{0.1,0.75,0.2}

\graphicspath{{./}} 

\definecolor{darkred}{rgb}{0.6,0.1,0.1}
\definecolor{darkgreen}{rgb}{0.1,0.6,0.1}
\definecolor{darkblue}{rgb}{0.1,0.1,0.6}

\newcommand{\be}{\begin{equation}}
\newcommand{\ee}{\end{equation}}
\newcommand{\bes}{\begin{equation*}}
\newcommand{\ees}{\end{equation*}}
\newcommand{\bfig}{\begin{figure}}
\newcommand{\efig}{\end{figure}}
\newcommand{\bt}{\begin{table}}
\newcommand{\et}{\end{table}}
\newcommand{\bc}{\begin{center}}
\newcommand{\ec}{\end{center}}

\newcommand{\mt}[1]{\mathrm{#1}}
\def\st{\, \left|\right. \,}
\def\:{\colon}

\def\wto{\rightharpoonup}
\newcommand\restrict[1]{\raisebox{-.1ex}{\footnotesize$|$}_{\raisebox{.2ex}{\scriptsize $#1$}}} 

\newcommand{\abs}[1]{\left\vert#1\right\vert}
\newcommand{\ap}[1]{\left\langle#1\right\rangle}
\newcommand{\norm}[1]{\left\Vert#1\right\Vert}

\def\grad{\nabla}

\DeclareMathOperator{\dive}{div}

\DeclareMathOperator{\supp}{supp}

\DeclareMathOperator{\dom}{dom}

\DeclareMathOperator{\inte}{int}

\def\R{\mathbb{R}} 
\def\N{\mathbb{N}} 
\def\P{\mathcal{P}} 
\def\M{\mathcal{M}} 

\def\e{\varepsilon}

\def\d{\,\mathrm{d}}
\def\ird{\int_{\R^d}}
\def\irdrd{\int_{\R^d\times\R^d}}

\def\XXint#1#2#3{{\setbox0=\hbox{$#1{#2#3}{\int}$ }
	\vcenter{\hbox{$#2#3$ }}\kern-.6\wd0}}

\def\der{\mathrm{d}}

\def\p{\partial}

\def\id{\mathrm{id}} 

\begin{document}
\author{Francesco S. Patacchini \and  Dejan Slep\v{c}ev}
\address{Francesco S. Patacchini -- IFP Energies nouvelles, 1 et 4 avenue de Bois-Pr\'eau, 92852 Rueil-Malmaison, France; Department of Mathematical Sciences, Carnegie Mellon University, Pittsburgh, PA 15213, USA}
\address{Dejan Slep\v{c}ev -- Department of Mathematical Sciences, Carnegie Mellon University, Pittsburgh, PA 15213, USA}
\email{francesco.patacchini@ifpen.fr}
\email{slepcev@math.cmu.edu}

\title{The nonlocal-interaction equation near attracting manifolds}

\date{\today}

\maketitle


\begin{abstract}
	We study the approximation of the nonlocal-interaction equation restricted to a compact manifold $\M$ embedded in $\R^d$, and more generally compact sets with positive reach (i.e. prox-regular sets).  We show  that the equation on $\M$ can be approximated by the classical nonlocal-interaction equation on $\R^d$ by adding an external potential which strongly attracts to $\M$. The proof relies on the Sandier--Serfaty approach \cite{SS,Serfaty} to the $\Gamma$-convergence of gradient flows.
	As a by-product, we recover well-posedness for the nonlocal-interaction equation on $\M$, which was shown \cite{CSW}.
	We also provide an another approximation to the interaction equation on $\M$, based on iterating
	approximately solving an interaction equation on $\R^d$ and 
	projecting to $\M$.  We show convergence of this scheme, together with an estimate on the rate of convergence. Finally, we conduct numerical experiments, for both the attractive-potential-based and the projection-based approaches, that highlight the effects of the geometry on the dynamics.
\end{abstract}

\section{Introduction}

We consider a continuum first-order model for the nonlocal-interaction of agents constrained to move within a compact manifold $\M$ embedded in Euclidean space $\R^d$. 
While the locations of the agents are restricted to $\M$ the interaction forces act in the ambient space. 

Although for simplicity we shall often call $\M$ a manifold, our setup actually allows for arbitrary sets with positive reach which include  manifolds with boundaries and with outside corners.
When $\M$ has full dimension, what we are studying is in fact an aggregation model with no-flux boundary conditions. The interaction is modeled via a smooth pairwise potential which can feature distinct regimes of repulsion and attraction, depending on the Euclidean distance that separates any pair of agents. The setting in which we consider our solutions and equations is that of gradient flows (curves of maximal slope) in the spaces of probability measures endowed with Wasserstein metric. We base our work  on the theory developed in \cite{AGS,SS,Serfaty}.

Let us describe further our problem and motivate it by reviewing recent relevant works. We then present our notation, assumptions and main results.

\subsection{Description of the problem and motivation}

The nonlocal-interaction equation on a compact manifold $\M$ embedded in $\R^d$ that we study is given by
\be\label{eq:main}
	\begin{cases}
	    \p_t \rho + \dive(\rho u) = 0,\\ u = P_\M(-\grad W * \rho),
	\end{cases} \quad \text{on $\M$},
\ee
where $W\: \R^d \to \R$ is an interaction potential and $P_\M$ is the generalization of the projection on the tangent space of $\M$, defined precisely in \eqref{eq:projection}.
Here, the unknown $\rho\in\P(\M)$ is a Borel probability measure supported on $\M$. When $\M$ is $d$-dimensional, this problem is the interaction equation on $\M$ with no-flux boundary conditions. 

Note that the gradient and convolution operators are Euclidean, i.e., with respect to the ambient space $\R^d$, and not intrinsic to $\M$. In particular, the convolution in \eqref{eq:main} is given by
\be\label{eq:convolution}
	W*\rho(x) = \int_\M W(x-y)\d\rho(y) \quad \text{for all $x\in\M$}.
\ee
This is an important point to bear in mind as it means that \eqref{eq:main} is a mixed formulation, where any two point masses on $\M$ "see" each other and interact according to the  Euclidean distance while their motion is restricted to the manifold (via the projection operator). This mixed formulation \eqref{eq:main} has already been studied in \cite{CSW,WuSle15}. In \cite{WuSle15} the authors observe that, when $\M$ is $d$-dimensional and has $C^2$ boundary, this formulation is the gradient flow of the interaction energy
\be\label{eq:energy}
    E(\rho) = \begin{cases} \frac12\int_\M\int_\M W(x-y) \d\rho(y)\d\rho(x) & \text{if $\rho\in\P(\M)$},\\ \infty & \text{otherwise}, \end{cases}
\ee
on the space $\P(\M)$ endowed with Wasserstein metric. In particular the steepest descent vector, denoted here $-\mathrm{grad} E(\rho)$ at a given configuration $\rho \in P(\M)$, satisfies
\be\label{eq:grad-descent}
    -\mathrm{grad}\, E(\rho) = P_\M(-\grad W* \rho).
\ee

We remark that 
adding a mobility matrix $A$ in \eqref{eq:main}, which models the space heterogeneity,  can be done without difficulty, changing \eqref{eq:grad-descent} into $-\mathrm{grad}\, E(\rho) = P_\M(-A\grad W* \rho)$; see \cite{WuSle15}.  In \cite{WuSle15}, the tangent vectors of the gradient flow formulation for the energy $E$ are vectors in $\R^d$ equipped with the associated Riemannian inner product. Note that, in a similar fashion, although without restriction to a subset, a Fokker--Planck equation in $\R^d$ with the mobility being the inverse of the metric tensor of a Riemannian manifold was studied in \cite{Lisini}.

In \cite{CSW} the authors extend the study of the well-posedness of this formulation to more general subsets of $\R^d$. One main difficulty with this extension is to use an appropriate notion of projection of vector fields on $\M$. When $\M$ is a smooth $d$-dimensional manifold as in \cite{WuSle15}, the projection of a vector $v\in\R^d$ at a point $x\in\M$ is the identity when $x\in \inte(\M)$ or $x\in \p\M$ and $v$ points into $\M$, i.e., $v$ belongs to the inward sector $T_x^\mathrm{in}\M$ of $T_x\M$ at $x$, and is the projection of $v$ to $T_x\p\M\subset T_x^\mathrm{in}\M$ when $x\in\p\M$ and $v$ does not point into $\M$. Note that when $x\in\inte(\M)$ we actually have $T_x^\mathrm{in}\M = T_x\M$ and $T_x^\mathrm{in}\M$ is thus a linear vector space; in general, for any $x\in\M$, we have $T_x^\mathrm{in}\M \subset T_x\M$. If we want to consider domains that are either lower-dimensional or present boundaries with corners, the set of tangent vectors "pointing into $\M$" (the inward sector) used to define the projection on a smooth manifold needs to be updated. To this end, the authors in \cite{CSW} extend the theory to prox-regular, or positive-reach, subsets of $\R^d$ for which inward tangent vectors are elements of Clarke tangent cones; see definitions in Section \ref{subsec:notation-assumptions}. Prox-regular sets extend significantly the variety of domains that can be studied; indeed, these include sets with outside corners and cusps in their boundaries. They can also be of dimension strictly less than $d$.
A  prox-regular set is a set with a tubular neighborhood whose every point has a unique closest point on the boundary of the set, i.e., a unique projection on the boundary. The radius of the largest such tubular neighborhood is called the reach of the set. Prox-regular sets are therefore also referred to as sets with positive reach. Note that convex sets are sets of infinite reach. Since we shall often refer to the reach of $\M$ in the following, we prefer here to use the terminology "positive-reach" over "prox-regular".

\subsubsection{Dynamics on $\M$ as a limit of dynamics on $\R^d$ with confining potential}
One goal of our paper, which we achieve in Section \ref{sec:result-epsilon}, is to show that solutions to the problem in \eqref{eq:main} on the set of positive reach $\M$ can be approximated by solutions to the following problem on all of $\R^d$:
\be\label{eq:main-epsilon}
	\begin{cases}
		\p_t \rho + \dive(\rho u) = 0,\\ u = -\grad W * \rho - \frac1\e \grad d_\M^2,
	\end{cases} \quad \text{on $\R^d$},
\ee
as the parameter $\e$ goes to $0$, where $d_\M$ is the distance to the set $\M$. The term $\tfrac1\e d_\M^2$ plays here the role of a confinement potential, making it expensive for the particles to lie at distances greater than $\sqrt{\varepsilon}$ away from $\M$.

A similar problem has been analyzed in \cite{ABC}, where the authors consider local diffusion and a confinement potential in addition to the nonlocal-interaction potential, and thus study a nonlinear, nonlocal Fokker--Planck equation on $\M$ with no-flux boundary conditions. Although their interaction potentials are allowed to be less regular than ours, the sets they consider have to be $d$-dimensional and connected, which is not our case. In this context, the authors show that the weak formulation for the Fokker--Planck equation on $\M$ can be approximated by the weak formulation on $\R^d$ obtained by continuously extending the confinement potential from $\M$ to all of $\R^d$ in such a way that the potential blows up outside of a tubular neighborhood of $\M$ as this neighborhood shrinks to $\p\M$. Furthermore, their diffusion term needs to be nonzero since their convergence analysis is based on $L^2$ estimates, which differs from our gradient flow approach allowing us to consider nonlocal interaction on its own. Indeed, by regarding our solutions as curves of maximal slope for the respective energies \eqref{eq:energy} and
\bes
    E_\e(\rho) = \frac12 \int_{\R^d}\int_{\R^d} W(x-y) \d\rho(y)\d\rho(x) + \frac1\e \int_{\R^d} d_\M^2(x)\d\rho(x) \quad \text{for all $\rho\in\P(\R^d)$},
\ees 
we are able to use the Sandier--Serfaty result \cite{SS,Serfaty} for the $\Gamma$-convergence of gradient flows to show the convergence of our approximating model \eqref{eq:main-epsilon} to the model on $\M$ given in \eqref{eq:main}.

\subsubsection{Dynamics on $\M$ as a limit of an iterative propagattion--projection scheme} 
An alternative approach to considering \eqref{eq:main-epsilon} in order to approximate solutions to \eqref{eq:main} is using the flow map of  the solution to the classical interaction equation on $\R^d$, i.e., 
\be\label{eq:main-rd}
\begin{cases}
    \p_t \mu + \dive(\mu u) = 0,\\ u = -\grad W * \mu,
\end{cases} \quad \text{on $\R^d$}.
\ee
We recall that this flow map is defined as the map $\Phi \: [0,\infty) \times \R^d \to \R^d$ such that, for all $x_0\in\R^d$, we have
\bes
\begin{cases}
    \frac{\der}{\der t} \Phi_t(x_0) = u(\Phi_t(x_0),t),\\
    \Phi_0(x_0) = x_0.
\end{cases}
\ees
Then, the solution $\mu$ to \eqref{eq:main-rd} is given by
\bes
    \mu(t) = (\Phi_t)_\#\rho^0 \quad \text{for all $t\geq 0$},
\ees
that is, the solution $\mu$ is the pushforward of the initial condition $\rho^0\in\P(\R^d)$ through the flow map $\Phi$; see \cite{AGS}. (Note that here, because of the nonlocality of the interaction velocity field, the flow map may depend on the solution itself so that the previous equation may in fact be an implicit formulation of the solution; we refer the reader to \cite{CCR,BLR} for the existence and uniqueness of explicit pushforward solutions to the interaction equation on $\R^d$.)

We define the approximation scheme by: for $n\geq0$ and small enough time step  $\tau>0$
    \bes
	    \begin{cases} \nu_0^\tau = \rho^0,\\
	    \nu_{n+1}^\tau = (\Pi_\M)_\# ((\Phi_\tau)_\#\nu_n^\tau) \end{cases}
    \ees
where $\Pi_\M$ is the projection to $\M$. The prox-regularity of $\M$ and choosing $\tau$ small enough ensures that the projection is unique. 
By interpolating we build a curve of probability measures supported on $\M$ which we show converges to a solution to \eqref{eq:main} as the size of the time step size vanishes. Our second goal in this paper is to show such convergence, which we do in Section \ref{sec:result-projected}.

Let us also remark that there is an interesting problem related to  \eqref{eq:main} where one studies the fully intrinsic version, where agents on $\M$ interact  according to the intrinsic metric of $\M$ rather than the ambient Euclidean space and therefore "see" each other along the manifold. In \cite{FPP,FHP,FZ19}, the well-posedness theory and long-time behavior (giving rise to asymptotic consensus), as well as numerical experiments, for the fully intrinsic model on the hemisphere, the hyperboloid and the special orthogonal group are carried out. There, the authors consider the case where the gradient and convolution operators are all defined with respect to the manifold $\M$; in particular, \eqref{eq:convolution} is replaced by
 \bes
 	W*_\mathrm{g}\rho(x) = \int_\M W(d_\mathrm{g}(x,y))\d\rho(y) \quad \text{for all $x\in\M$},
 \ees
 where $d_\mathrm{g}$ is the intrinsic, or geodesic, metric on $\M$. 
 There, the obstacles are the lack of convexity of $d_\mathrm{g}^2$ on subsets of $\M$ that are not geodesically convex and the difficulty to compare vectors in different tangent spaces in order to prove Lipschitz continuity of the velocity field for the resolution of the characteristic equations solved by flow maps. We also refer the reader to \cite{HK19,HKLN19,AHPS20} and the references therein for second-order swarming models on the sphere and the hyperboloid.

\subsection{Notation and assumptions}\label{subsec:notation-assumptions}

The functional context in which we study solutions to \eqref{eq:main} and \eqref{eq:main-epsilon} is that of probability measures. We denote by $\P(\R^d)$ the set of Borel probability measures on $\R^d$ and $\P_2(\R^d)$ the subset of $\P(\R^d)$ of measures with finite second moment. We  endow $\P_2(\R^d)$ with the \emph{(quadratic) Wasserstein distance} denoted by $d_2$: for all $\rho_0,\rho_1\in\P_2(\R^d)$ we define
\bes
    d_2(\rho_0,\rho_1) = \inf_{\pi\in\Pi(\rho_0,\rho_1)} \sqrt{\int_{\R^d\times\R^d} |x-y|^2 \d\pi(x,y) }, 
\ees
where $\Pi(\rho_0,\rho_1)$ is the set of transport plans from $\rho_0$ to $\rho_1$. That is, $\Pi(\rho_0,\rho_1)$ is the set of Borel probability measures on $\R^d\times\R^d$ with first marginal $\rho_0$ and second marginal $\rho_1$.

Let $\M\subset \R^d$ satisfy the following assumption:
\begin{hyp}
  \label{hyp:M}
    $\M$ is a compact subset of $\R^d$ with positive reach, denoted $\eta_\M$. 
\end{hyp}
\noindent For instance, $\M$ can be a manifold with $C^2$ boundary; it can also be a lower-dimensional object, such as a circle in $\R^2$ or even a manifold with outside corners as a rectangular sheet in $\R^3$. We write $\P(\M)$ the subset of $\P(\R^d)$ of measures supported in $\M$; since $\M$ is compact we actually have $\P(\M)\subset \P_2(\R^d)$. 

For all $x\in \M$, the projection operator $P_\M(x)$, used  in  \eqref{eq:main}, maps the vectors in $\R^d$ to ``tangent'' vectors to $\M$.  For  $v\in \R^d$ it is defined by
\be\label{eq:projection}
    P_\M(x)(v) = \{ w \in T_x^\mathrm{in}\M \st |v-w| = \inf_{u\in T_x^\mathrm{in}\M} |v-u| \},
\ee
where $T_x^\mathrm{in}\M$ is the \emph{Clarke tangent cone} of $\M$ at $x$, whose definition is
\begin{align*}
T_x^\mathrm{in}\M = \big\{ & v\in\R^d \st \forall\, (t_n)_n\to0^+,\; \forall\, \M\supset (x_n)_n \to x,\; \\
& \exists\, \R^d \supset (v_n)_n\to v,\; \te{ s.t. }\forall\, k,\, x_k + t_k v_k\in\M \big\}.
\end{align*}
So $P_\M(x)$ is the projection on the Clarke tangent cone at $x$. Because this tangent cone is always a closed and convex subset of $\R^d$, the projection in \eqref{eq:projection} is always a singleton, so that $P_\M(x) \: \R^d \to T_x^\mathrm{in}\M$. Using this projection in \eqref{eq:main}, we ensure that particles moving according to \eqref{eq:main} do not leave $\M$. We notice, as expected, that when $\M$ is a $C^1$ manifold without boundary the Clarke tangent cone coincides with the classical linear tangent space, in which case the above projection can be rewritten as
\bes
    P_\M(x)(v) = \begin{cases} v & \text{if $v\in T_x\M$},\\ 
    \Pi_{T_x\M}(v) & \text{otherwise}, \end{cases}
\ees
where $\Pi_{T_x\M}\: \R^d\to T_x\M$ is the projection on $T_x\M$, since in this case $T_x^\mathrm{in}\M = T_x\M$.

For any given functional $F\: \P(\R^d) \to (-\infty,\infty]$ we write $\dom(F)$ its domain, i.e., $\dom(F) = \{ \mu \in \P(\R^d) \st F(\mu) < \infty\}$, and we say that $F$ is proper if $\dom(F) \neq \emptyset$. We define the \emph{interaction energy} $E\: \P_2(\R^d) \to (-\infty,\infty]$ for all $\rho \in \P_2(\R^d)$ by
\be\label{eq:energy-manifold}
	E(\rho) = \begin{cases} \frac12 \int_\M W* \rho(x) \d\rho(x) & \mbox{if $\rho \in \P(\M)$}, \\ \infty & \mbox{otherwise}, \end{cases}
\ee
where $W\: \R^d \to \R$ is an interaction kernel verifying the assumption below:
\begin{hyp}
  \label{hyp:W}
    $W$ is of class $C^2$, is symmetric, semiconvex, bounded from below, and has at-most-quadratic growth at infinity.
\end{hyp}
\noindent Because $W$ is continuous and $\M$ is compact, the domain of $E$ satisfies $\dom(E) = \P(\M)$. We also define, for all $\e>0$, the \emph{$\e$-interaction energy} $E_\e \: \P_2(\R^d) \to (-\infty,\infty]$ by
\be\label{eq:energy-epsilon}
	E_\e(\rho) = \frac12 \ird W* \rho(x) \d\rho(x) + \frac1\e \ird d_\M^2(x) \d\rho(x) \quad \mbox{for all $\rho \in \P_2(\R^d)$},
\ee
where $d_\M(x)$ denotes the distance of any point $x\in\R^d$ to $\M$. As already mentioned above, the $\e$-dependent part of this energy plays the role of a confinement potential. Because $W$ has at-most-quadratic growth at infinity and $\ird d_\M^2(x) \d \rho(x) < \infty$ for all $\rho \in \P_2(\R^d)$, we have $\dom(E_\e) = \P_2(\R^d)$. Note that the condition on the at-most-quadratic growth at infinity on $W$ is not restrictive our case, since our interest lies in what happens within the compact set $\M$ or within bounded regions containing $\M$.

We  study \eqref{eq:main} and its approximation \eqref{eq:main-epsilon} as the gradient flows for the respective energies \eqref{eq:energy-manifold} and \eqref{eq:energy-epsilon}. We refer the reader to \cite{AGS} for an extensive theory of gradient flows, and only introduce the definitions we shall need here in order to prove our main results. We write $AC^2([0,\infty);\P_2(\R^d))$ the set of $2$-absolutely continuous curves from $[0,\infty)$ to $\P_2(\R^d)$, that is, we write $\rho \in AC^2([0,\infty);\P_2(\R^d))$ if there exists a function $m\in L^2(0,\infty)$ such that 
\bes
    d_2(\rho(s),\rho(t)) \leq \int_s^t m(r) \d r \quad \text{for all $0\leq s\leq t$}.
\ees
Given $F\colon \P_2(\R^d)\to(-\infty,\infty]$ proper, a function $g\colon\P_2(\R^d)\to[0,\infty]$ is said to be a \emph{strong upper gradient} for $F$ if for every $\rho\in AC^2([0,\infty);\P_2(\R^d))$ the map $g\circ\rho$ is Borel measurable and
\bes
    |F(\rho(t))-F(\rho(s))| \leq \int_s^t g(\rho(\tau))|\rho'|(\tau) \d\tau \quad \mbox{for all $0\le s\le t$},
\ees
where $|\rho'|$ is the \emph{metric derivative} of $\rho$, that is,
\bes
    |\rho'|(t) = \lim_{h\to0}\frac{d_2(\rho(t),\rho(t+h))}{|h|} \quad\mbox{for all $t\geq0$},
\ees
and a curve $\rho \in AC^2([0,\infty);\P_2(\R^d))$ is a \emph{curve of maximal slope} for $F$ with respect to its strong upper gradient $g$ if and only if $t\mapsto F(\rho_t)$ is nonincreasing and
\bes
	F(\rho(t))-F(\rho(s))+\frac{1}{2}\int_s^t \left( g(\rho(\tau))^2+|\rho'|(\tau)^2 \right) \d\tau \leq 0 \quad \mbox{for all $0\le s\le t$}.
\ees
Also, the \emph{local slope} of $F$ is defined as
\bes
	|\partial F|(\rho) = \begin{cases} \displaystyle \limsup_{\mu \to \rho} \frac{(F(\rho) - F(\mu) )_+}{d_2(\rho,\mu)} & \mbox{for all $\rho\in \dom(F)$},\\ \infty & \mbox{otherwise}, \end{cases}
\ees
where the subscript $+$ denotes the positive part. It can be checked that strong upper gradients for the interaction energies $E$ and $E_\e$, for any $\e>0$, are given by their respective local slopes $|\p E|$ and $|\p E_\e|$, which are also semiconvex and narrowly lower semicontinuous thanks to our assumptions on the interaction kernel $W$; see \cite{AGS}. Recall that the narrow topology on $\P(\R^d)$ is given by the following definition: we say that a sequence $(\rho^n)_n\subset \P(\R^d)$ converges \emph{narrowly} to some $\rho\in\P(\R^d)$ if 
\bes
    \int_{\R^d} f(x) \d\rho^n(x) \to \int_{\R^d} f(x) \d\rho(x) \quad \mbox{as $n\to\infty$ for all $f\colon \R^d\to\R$ continuous and bounded.}
\ees
Following the Sandier--Serfaty theory, we consider gradient flows as curves of maximal slopes with respect to local slopes. Accordingly, we have the following definition:
\begin{defn}[gradient flow]
    We say that a curve in $AC^2([0,\infty);\P_2(\R^d))$ is a
    \emph{gradient flow} for $E$ (respectively, $E_\e$) if it is a curve of maximal slope with respect to $|\p E|$ (respectively, $|\p E_\e|$). For convenience, we shall sometimes refer to gradient flows for $E_\e$ as \emph{$\e$-gradient flows}.
\end{defn}

\subsection{Main results}

The first main result establishes that  the gradient flow of $E_\e$, i.e., \eqref{eq:main-epsilon}, converges as $\e \to 0$ to the gradient flow of $E$, i.e., \eqref{eq:main}. 

\begin{thm}[$\e$-gradient flow scheme]\label{thm:main-epsilon}
Consider$\M$ satisfying Assumption \ref{hyp:M} and a potential $W$ satisfying  Assumption \ref{hyp:W}. Let $\rho^0 \in \P(\M)$.
	For $\e>0$, let $\rho_\e \in AC^2([0,\infty);\P_2(\R^d))$ be an $\e$-gradient flow such that $\rho_\e(0) = \rho_\e^0$ for some $\rho_\e^0 \in \P_2(\R^d)$. Assume that $d_2(\rho_\e^0,\rho^0) \to0$ and $\limsup E_\e(\rho_\e^0) \leq E(\rho^0)$ as $\e\to0$. Then there exists $\rho \in AC^2([0,\infty);\P_2(\R^d))$ such that $\rho(0) = \rho^0$, $\rho(t) \in \P(\M)$ and
%
$d_2(\rho_\e(t),\rho(t)) \to0$ as $\e \to0$ for all $t\geq0$, and such that $\rho$ is a gradient flow for $E$.
\end{thm}

Theorem \ref{thm:main-epsilon} contains a compactness part, in which we prove that any $\e$-gradient flow has a limit, and a convergence part, where we prove that this limit is indeed a gradient flow for $E$. We prove the convergence part of the main result using the Sandier--Serfaty approach, which we recall in Theorem \ref{thm:ss}.
We note that while each of the energies $E_\e$ is semiconvex, the semiconvexity diverges to $-\infty$ as $\e \to 0$, so the convergence does not follow directly from stability of gradient flows. 
On the other hand, $\e$-gradient flows exist thanks to the semiconvexity and quadratic growth at infinity of the interaction potential $W$ (cf. \cite[Corollary 3.2]{DSNotes} for instance). Hence the above theorem ensures the existence of gradient flows for the energy in $E$ given in \eqref{eq:energy}. This fact was already proved in \cite{CSW} via a more classical tool, namely differential inclusion theory. For completeness here, we recall a stability result from \cite{CSW} in Proposition \ref{prop:stability}, which proves uniqueness.
\medskip

The second question we answer  is whether one can use the full-space equation \eqref{eq:main-rd} to approximate the projected version \eqref{eq:main}. We show this is indeed the case:
\begin{thm}[projected gradient flow scheme]\label{thm:main-projected}
Consider$\M$ satisfying Assumption \ref{hyp:M} and a potential $W$ satisfying  Assumption \ref{hyp:W}. 
    Let $\rho^0\in\P(\M)$. For all $t\geq0$, let $\Phi_t\: \R^d \to \R^d$ be the flow map associated to the $\R^d$ nonlocal-interaction equation \eqref{eq:main-rd} and let $\rho\in AC^2([0,\infty),\P(\M))$ be the gradient flow for $E$, i.e. the solution of \eqref{eq:main}, with initial condition $\rho(0) = \rho^0$. Define, for any positive integer $n$ and time step  $\tau>0$ small enough,
    \bes
	    \begin{cases} \nu_0^\tau = \rho^0,\\
	    \nu_{n+1}^\tau = (\Pi_\M)_\# ((\Phi_\tau)_\#\nu_n^\tau), \end{cases}
    \ees
    where $r<\eta_\M$ and $\Pi_\M\:\M_r\to\M$ is the projection on the set $\M$ from the $r$-neighborhood $\M_r:=\{x\in\R^d\st d_\M(x)<r\}$ of $\M$. Define the interpolation
    \bes
	    \begin{cases} \nu^\tau(0) = \rho^0,\\
	    \nu^\tau(t) = (\Pi_\M)_\# ((\Phi_{t-n\tau})_\# \nu_n^\tau) & \mbox{for all $t\in(n\tau,(n+1)\tau]$}. \end{cases}
    \ees
    Then, for all $t\geq0$ we have $d_2(\nu^\tau(t),\rho(t)) \to 0$ as $\tau\to0$.
\end{thm}
The proof of Theorem \ref{thm:main-projected} is based on an estimate of the rate of convergence of the approximating sequence $(\nu^\tau)_\tau$ to the solution $\rho$ to \eqref{eq:main}, which turns out to be linear in the time step size $\tau$. Furthermore, what is meant by any time step size ``small enough'' will be made clear in the proof.

In the following, we present the proofs of Theorems \ref{thm:main-epsilon} and \ref{thm:main-projected} in Sections \ref{sec:result-epsilon} and \ref{sec:result-projected}, respectively, and our numerical experiments in Section \ref{sec:exp}. Our code is fully accessible on our \href{https://github.com/francesco-patacchini/interaction-equation-attracting-manifolds}{GitHub} repository \cite{GitHub}.

\section{$\e$-gradient flow scheme (proof of Theorem \ref{thm:main-epsilon})}\label{sec:result-epsilon}

\subsection{$\Gamma$-convergence}

Before giving the core of the proof of Theorem \ref{thm:main-epsilon}, we show the $\Gamma$-convergence of the $\e$-energy $E_\e$ towards $E$ as $\e\to0$. This result is used throughout the paper.
\begin{lem}
  \label{lem:gamma-conv}
	The sequence $(E_\e)_\e$ $\Gamma$-converges to $E$ with respect to the narrow topology as $\e \to0$.
\end{lem}

\begin{proof}
We first show the liminf inequality and then the limsup inequality.

\medskip
\emph{Step 1: liminf inequality.} Let $\rho \in \P_2(\R^d)$ and $(\rho_\e)_\e$ be a sequence in $\P_2(\R^d)$ such that $\rho_\e \wto \rho$ narrowly as $\e\to0$. We want to show
\bes
	\liminf_{\e\to0} E_\e(\rho_\e) \geq E(\rho).
\ees
Suppose first that $\supp(\rho) \subset \M$. We have $E_\e(\rho_\e) \geq \frac12 \ird W*\rho_\e(x) \d \rho_\e(x)$, by the continuity and boundedness from below of $W$ and the Portmanteau theorem (cf. \cite[Theorem 2.1]{Billingsley}),
\bes
	\liminf_{\e \to 0} \ird W*\rho_\e(x) \d \rho_\e(x) \geq \ird W*\rho(x) \d \rho(x) = \int_{\M} W*\rho(x) \d \rho(x).
\ees
Hence,
\bes
	\liminf_{\e \to 0} E_\e(\rho_\e) \geq \frac12 \int_{\M} W*\rho(x) \d \rho(x) = E(\rho).
\ees
Suppose now that $\supp(\rho) \not \subset \M$. We have $E_\e(\rho_\e) \geq \frac1\e \ird d_\M(x)^2 \d \rho_\e(x)$ because $W\geq0$. Also, by continuity and boundedness from below of $d_\M^2$, we have, by the Portmanteau theorem,
\bes
	\liminf_{\e\to0} \ird d_\M(x)^2 \d \rho_\e(x) \geq \ird d_\M(x)^2 \d \rho(x) > 0,
\ees
which yields $\tfrac1\e \ird d_\M(x)^2 \d \rho_\e(x) \to \infty$ as $\e\to0$, and so
\bes
	\lim_{\e\to0} E_\e(\rho_\e) = \infty= E(\rho).
\ees

\medskip
\emph{Step 2: limsup inequality.} Let $\rho \in \P_2(\R^d)$. We want to show that there exists $(\rho_\e)_\e$, a sequence in $\P_2(\R^d)$, such that $\rho_\e\wto \rho$ narrowly as $\e\to0$ and
\bes
	\limsup_{\e\to0} E_\e(\rho_\e) \leq E(\rho).
\ees
Choose the constant sequence $(\rho_\e)_\e$ given by $\rho_\e = \rho$ for all $\e>0$. Suppose first that $\supp(\rho) \subset \M$. Then $\ird d_\M(x)^2 \d \rho_\e(x) = 0$ for all $\e>0$. Therefore $E_\e(\rho_\e) = E(\rho)$
which trivially implies the limsup inequality. Suppose now that $\supp(\rho) \not\subset \M$. As for the liminf inequality we conclude that
\bes
	\lim_{\e\to0} E_\e(\rho_\e) = \infty= E(\rho). \qedhere
\ees
\end{proof}

As an interesting corollary of the $\Gamma$-convergence we have the following.
\begin{lem}\label{lem:mins}
	For all $\e>0$ there exists a minimizer $\rho_\e^*$ of $E_\e$. Moreover, $\rho_\e^*$ converges in Wasserstein metric, along a subsequence, to some $\rho^*\in \P(\M)$, which is a minimizer of $E$.
\end{lem}
\begin{proof}
	Let us first prove the existence of $\rho_\e^*$ for all $\e>0$. Let $\e>0$ and let $(\rho_\e^n)_n \subset \P_2(\R^d)$ be a minimizing sequence for $E_\e$. Because $E_\e \leq E$ is proper, we then know there exists $A_\e>0$, independent of $n$, such that $E_\e(\rho_\e^n) \leq A_\e$ for all $n\in\N$ large enough. Choose any such $n$. Note that by the definition of $d_\M$ we have $d_\M(x) + L_\M \geq |x|$, so that $d_\M(x)^2 \geq |x|^2 - 2L_\M |x|$ for all $x\in\R^d$, where $L_\M$ is the diameter of $\M$. Therefore, by the nonnegativity of $W$ and Jensen's inequality,
\begin{align*}
	A_\e \geq E_\e(\rho_\e^n) &\geq \frac1\e \left( \ird |x|^2 \d\rho_\e^n(x) -2L_\M \ird |x| \d \rho_\e^n(x) \right)\\
	&\geq \frac1\e \left( \ird |x|^2 \d\rho_\e^n(x) -2L_\M \sqrt{\ird |x|^2 \d \rho_\e^n(x)} \right)\\
	&\geq \frac1\e \left (\sqrt{\ird |x|^2 \d\rho_\e^n(x)} - 2L_\M \right)^2 - \frac{4L_\M^2}{\e}.
\end{align*}
This shows that the second moment of $\rho_\e^n$ is uniformly bounded in $n$. By Prokhorov's theorem we deduce the existence of $\rho_\e^* \in \P_2(\R^d)$ such that, up to a subsequence, $\rho_\e^n \wto \rho_\e^*$ narrowly as $n\to\infty$. By narrow lower semicontinuity of $E_\e$ (since $W$ and $d_\M$ are continuous and bounded from below; see again the Portmanteau theorem) we therefore get that $\rho_\e^*$ is a minimizer of $E_\e$. 
	
Let us now prove the existence of $\rho^*$ so that $\rho_\e^* \wto \rho^*$ narrowly as $\e\to0$. Pick $\mu \in \P(\M)$. By the $\Gamma$-convergence of $(E_\e)_\e$ and the minimality of $(\rho_\e^*)_\e$ we know there exists $(\mu_\e)_\e \subset \P_2(\R^d)$ such that $\mu_\e \wto \mu$ narrowly as $\e\to0$ and
\bes
	\limsup_{\e\to0} E_\e(\rho_\e^*) \leq \limsup_{\e\to0} E_\e(\mu_\e) \leq E(\mu) < \infty.
\ees
Therefore, there exists $A_0>0$, independent of $\e$, such that $E_\e(\rho_\e^*) \leq A_0$ for all $\e>0$ small enough. Choose any such $\e$. By a similar argument as above, assuming without loss of generality that $\e<1$, we have
\bes
	A_0 \geq E_\e(\rho_\e^*) \geq \left (\sqrt{\ird |x|^2 \d\rho_\e(x)} - 2L_\M \right)^2 - 4L_\M^2,
\ees
from which we get the existence of $M_0>0$ such that
\begin{equation}
  \label{eq:moment-bound-0}
  \ird \abs{x}^2 \d\rho_\e(t,x) \leq M_0 \quad \mbox{for all $t\geq0$}.
\end{equation}
This shows that the second moment of $\rho_\e^*$ is bounded from above uniformly in $\e$. By Prokhorov's theorem we deduce the existence of $\rho^* \in \P_2(\R^d)$ such that $\rho_\e^* \wto \rho^*$ as $\e\to0$ along a subsequence. By narrow $\Gamma$-convergence, we obtain
\bes
	E(\rho^*) \leq \liminf_{\e\to0} E_\e(\rho_\e^*) \leq \limsup_{\e\to0} E_\e(\rho_\e^*) \leq \limsup_{\e\to0} E_\e(\mu_\e) \leq E(\mu) < \infty.
\ees
Hence $\rho^*\in \dom(E) = \P(\M)$ and $\rho^*$ is a minimizer of $E$. Finally, by \eqref{eq:moment-bound-0} and \cite[Proposition 7.1.5]{AGS}, we also conclude that $d_2(\rho_\e^*,\rho^*)\to0$.
\end{proof}

\subsection{Compactness}

We give here the proof of the compactness part of Theorem \ref{thm:main-epsilon}.

\begin{lem}\label{lem:compactness}
	Let $(\rho_\e)_\e$ and $\rho^0$ be as in Theorem \ref{thm:main-epsilon}. Then there is $\rho \in AC^2([0,\infty);\P_2(\R^d))$ such that $\rho(0) = \rho^0$, $\rho(t) \in \P(\M)$ for all $t\geq0$ and, along a subsequence,
\bes
	d_2(\rho_\e(t),\rho(t)) \to0  \quad \; \te{for all } t \geq 0 \quad \te{as } \e \to0.
\ees
\end{lem}
We remark that once we show the remainder of Theorem \ref{thm:main-epsilon} we will know that $\rho$ is a solution to an initial-value problem for \eqref{eq:main-epsilon}. Since then, by Proposition \ref{prop:stability}, the solutions are unique, we will conclude that in fact $d_2(\rho_\e(t),\rho(t)) \to 0$ as $\e \to 0$ in Lemma \ref{lem:compactness}, not just along a subsequence.
\begin{proof}[Proof of Lemma \ref{lem:compactness}]
Since $\limsup E_\e(\rho_\e^0) \leq E(\rho^0) < \infty$ as $\e \to0$ there exists a constant $A_0>0$, independent of $\e$ and $t$, such that $E_\e(\rho_\e^0) \leq A_0$ for all $\e>0$ small enough. Choose any such $\e$. By decreasing monotonicity of the energy $E_\e$ along $\rho_\e$ (because $\rho_\e$ is a curve of maximal slope for $E_\e$ with respect to the strong upper gradient $|\p E_\e|$), this implies
\bes
	E_\e(\rho_\e(t)) \leq A_0 \quad \mbox{for all $t\geq0$}.
\ees
By a similar argument to that given in the proof of Lemma \ref{lem:mins}, letting again $\e < 1$, we have
\bes
	A_0 \geq E_\e(\rho_\e(t)) \geq \left (\sqrt{\ird |x|^2 \d\rho_\e(t,x)} - 2L_\M \right)^2 - 4L_\M^2,
\ees
where we recall that $L_\M$ is the diameter of $\M$. This proves there exists $M_0>0$, independent of $\e$ and $t$, such that
\be\label{eq:moment-bound}
	\ird |x|^2 \d\rho_\e(t,x) \leq M_0 \quad \mbox{for all $t\geq0$}.
\ee
Now, by the evolution variational inequality \cite[Theorem 5.3(iii)]{AGS} thanks to the semiconvexity of $W$, we have, for all $0\leq \tau \leq t$,
\begin{align*}
	d_2^2(\rho_\e(t),\rho_\e(\tau)) &\leq 2\int_\tau^t (E_\e(\rho_\e(\tau)) - E_\e(\rho_\e(s))) \d s\\
	&\leq 2\int_\tau^t (E_\e(\rho_\e^0) - E_\e(\rho_\e(t))) \d s\\
	&\leq 2 E_\e(\rho_\e^0) (t-\tau) \leq 2A_0 (t-\tau),
\end{align*}
where we used that $E_\e\geq0$; by swapping $t$ and $\tau$ when $0\leq t < \tau$, we get
\bes
	d_2^2(\rho_\e(t),\rho_\e(\tau)) \leq 2A_0 |t-\tau| \quad \mbox{for all $t,\tau \geq0$}.
\ees
This, together with \eqref{eq:moment-bound}, Prokhorov's theorem and the Arzel\`a--Ascoli theorem as given in \cite[Proposition 3.3.1]{AGS}, implies that there exists $\rho \in C([0,\infty];\P_2(\R^d))$ such that, up to a subsequence, $\rho_\e(t) \wto \rho(t)$ narrowly as $\e\to0$ for all $t\geq0$. By \eqref{eq:moment-bound} and \cite[Proposition 7.1.5]{AGS} it follows that, in fact, $d_2(\rho_\e(t),\rho(t)) \to0$ as $\e\to0$ for all $t\geq0$. To show that $\rho(0) = \rho^0$ we simply use the triangle inequality:
\bes
	d_2(\rho(0),\rho^0) \leq d_2(\rho(0),\rho_\e^0) + d_2(\rho_\e^0,\rho^0) = d_2(\rho(0),\rho_\e(0)) + d_2(\rho_\e^0,\rho^0) \to 0 \quad \mbox{as $\e\to0$.}
\ees
By the narrow $\Gamma$-convergence of $(E_\e)_\e$ we deduce that 
\bes
	\infty > A_0 \geq \liminf_{\e\to0} E_\e(\rho_\e(t)) \geq E(\rho(t)) \quad \mbox{for all $t\geq0$},
\ees
which shows that $\rho(t) \in \dom(E) = \P(\M)$ for all $t\geq0$.

We finally need to show that $\rho$ is actually a curve in $AC^2([0,\infty);\P_2(\R^d))$. This part is based on \cite[Theorem 5.6]{CT}; see also \cite[Lemma 4.3]{CPSW}. Fix any $t \geq0$. We have
\bes
	\int_0^t |\rho_\e'|(s)^2 \d s = E_\e(\rho_\e^0) - E_\e(\rho_\e(t)) \leq A_0.
\ees
Then, up to a subsequence, $\lim_{\e\to0} \int_0^t |\rho_\e'|(s)^2 \d s = C$ for some $t$-independent $C\geq 0$. Therefore $|\rho_\e'|$ is bounded in $L^2([0, t])$ and so, up to a further subsequence, it is $L^2$-weakly convergent to some $v \in L^2([0,t])$. It is then also $L^1$-weakly convergent to $v$, so that
\bes
	\lim_{\e\to0} \int_{\tau}^{t} |\rho_\e'|(s) \d s = \int_{\tau}^{t} v(s) \d s \quad \mbox{for all $0\leq \tau \leq t$}.
\ees
We also know that, by definition of the metric derivative and $\rho_\e$ being $2$-absolutely continuous,
\bes
	d_2(\rho_\e(\tau), \rho_\e(t)) \leq \int_{\tau}^{t} |\rho_\e'|(s) \d s.
\ees
Then, by the narrow lower semicontinuity of $d_2$ (see \cite[Proposition 2.5]{AG}), sending $\e\to0$ yields
\be\label{eq:d2-lsc}
	d_2(\rho(\tau), \rho(t)) \leq \int_{\tau}^{t} v(s) \d s,
\ee
which implies that $\rho \in AC^2 ([0,\infty);\P_2(\R^d))$.
\end{proof}

\subsection{Convergence}

We now present the proof of the convergence part of Theorem \ref{thm:main-epsilon}. Let us recall the result by Sandier--Serfaty on which we base our proof:
\begin{thm}[Sandier--Serfaty]\label{thm:ss}
	For all $\e>0$, let $\rho_\e \in AC^2([0,\infty);\P_2(\R^d))$ be an $\e$-gradient flow such that $\rho_\e(0) = \rho_\e^0$ for some $\rho_\e^0 \in \P_2(\R^d)$. Assume that $d_2(\rho_\e(t),\rho(t)) \to 0$ as $\e\to0$ for all $t \geq0$ for some $\rho\in AC^2([0,\infty);\P_2(\R^d))$ such that $\rho^0:=\rho(0) \in \P(\M)$ and $\limsup E_\e(\rho_\e^0) \leq E(\rho^0)$ as $\e\to0$. Furthermore, suppose that the following conditions hold for all $t\geq0$:
\begin{enumerate}[label=(C\arabic*)]
	\item \label{cond:md} $\displaystyle \liminf_{\e\to 0} \int_0^t |\rho_\e'|_{d_2}(s)^2 \d s \geq \int_0^t|\rho'|_{d_2}(s)^2\d s$.\\
	\item \label{cond:liminf} $\displaystyle \liminf_{\e\to0} E_\e(\rho_\e(t)) \geq \displaystyle E(\rho(t))$.\\
	\item \label{cond:slopes} $\displaystyle \liminf_{\e\to 0} |\p E_\e|(\rho_\e(t)) \geq |\p E|(\rho(t))$.
\end{enumerate}
Then $\rho$ is a gradient flow for $E$ starting from $\rho^0$, and
\be\label{eq:serfaty-conseq}
	\begin{cases}
		|\rho_\e'| \xrightarrow[\e\to0]{} |\rho'| & \mbox{in $L^2([0,\infty))$},\\
		E_\e(\rho_\e(t)) \xrightarrow[\e\to0]{} E_\e(\rho(t)) & \mbox{for all $t\geq0$},\\
		|\p E_\e|(\rho_\e) \xrightarrow[\e\to0]{} |\p E|(\rho) & \mbox{in $L^2([0,\infty))$}.
	\end{cases}
\ee
\end{thm}

In our case, Condition \ref{cond:md} follows directly from the above argument. Indeed, writing $\rho$ and $v$ as in the proof of Lemma \ref{lem:compactness}, by \eqref{eq:d2-lsc} and \cite[Theorem 1.1.2]{AGS} we have $|\rho_\e'|(s) \leq v(s)$ for almost every $s\geq0$. Then, by the weak lower semicontinuity of the $L^2$-norm, we get
\bes
	\liminf_{\e\to0} \int_0^t |\rho_\e'|(s)^2 \d s \geq \int_0^t v(s)^2 \d s \geq \int_0^t |\rho'|(s)^2 \d s,
\ees
which is \ref{cond:md}. Condition \ref{cond:liminf} is a direct consequence of the $\Gamma$-convergence of $(E_\e)_\e$ to $E$ given in Lemma \ref{lem:gamma-conv}. Thus, we are only left with proving Condition \ref{cond:slopes}, which we do in the remainder of this section.

\subsubsection*{\textit{\textbf{Proof of Condition \ref{cond:slopes}}}} Let $(\rho_\e)_\e$ and $\rho$ be as in Theorem \ref{thm:main-epsilon}. For clarity in this proof we will often omit time dependence so that, for example, we will write $\rho_\e$ and $\rho$ in place of $\rho_\e(t)$ and $\rho(t)$, respectively. 

Let us introduce some notation we shall use throughout the proof. For any $\mu \in \P(\R^d)$ and any Borel set $A\subset \R^d$ we write $\mu\restrict{A}$ the restriction of $\mu$ to $A$. We define, for all $i \in \{1,2\}$ and $j,k \in \{1,2,3\}$, the projections
\bes
\begin{gathered}
	\pi_i\: (\R^d)^2 \to \R^d,\, (x_1,x_2) \mapsto x_i,\\
	\tilde \pi_j\: (\R^d)^3 \to \R^d,\, (x_1,x_2,x_3) \mapsto x_j,\\
	\tilde \pi_{jk} \: (\R^d)^3 \to (\R^d)^2,\, (x_1,x_2,x_3) \mapsto (x_j,x_k).\\
\end{gathered}
\ees
We use the two identity maps $\id \: \R^d \to \R^d,\, x \mapsto x$, and $\id_2 \: \R^d \to (\R^d)^2,\, x\mapsto (x,x)$. We fix $r\in(0,\eta_\M)$ and write $\M_r$ the $r$-neighborhood of $\M$, i.e.,
\bes
	\M_r = \{ x\in \R^d \st d_\M(x) < r\}.
\ees 
We write 
\bes	
	\Pi_\M\: \M_r \to \M
\ees
the projection on $\M$ restricted to $\M_r$, which is well-defined for $r$ is smaller than the reach $\eta_\M$ of $\M$. Without losing generality we may suppose $r$ is such that $p_\e:=\rho_\e(\M_r)>0$ for all $\e>0$; note that $p_\e \to 1$ as $\e\to0$. 

For every $\e>0$ let us write $\hat \rho_\e \in \P_2(\R^d)$ the projection of $\rho_\e$ defined as
\bes
	\hat \rho_\e = (\Pi_\M)_\# \sigma_\e,
\ees
where $\sigma_\e\in \P_2(\R^d)$ stands for $(\rho_\e\restrict{\M_r})/p_\e$. We clearly have $\supp(\hat \rho_\e) \subset \M$ and so $\hat \rho_\e \in \dom(E)$. Then, by definition, for all $\e>0$ we have
\bes
	|\p E|(\hat \rho_\e) = \limsup_{\substack{\nu_\e \to \hat \rho_\e \\ \supp(\nu_\e) \subset \M}} \frac{\left(E(\hat \rho_\e) - E(\nu_\e)\right)_+}{d_2(\hat \rho_\e,\nu_\e)}.
\ees
Because $|\p E|$ is $d_2$-lower semicontinuous \cite[Corollary 2.4.10]{AGS} and since $d_2(\hat \rho_\e,\rho) \to0$ as $\e\to0$, we get
\bes
	\liminf_{\e \to 0} |\p E|(\hat\rho_\e) \geq |\p E|(\rho).
\ees
Therefore, Condition \ref{cond:slopes} is a direct consequence of
\be\label{eq:ls-result}
	\liminf_{\e\to0} |\p E_\e|(\rho_\e) \geq \liminf_{\e \to0} |\p E|(\hat \rho_\e).
\ee
Proving \eqref{eq:ls-result} will therefore conclude the proof of Condition \ref{cond:slopes}. The proof of the following lemma includes all the main ideas to achieve this; it shows \eqref{eq:ls-result} in the simpler case where $\supp(\rho_\e) \subset \M_r$ for all $\e>0$, which we later relax.

\begin{lem}\label{lem:main}
	Let $(\rho_\e)_\e$ be as in Theorem \ref{thm:main-epsilon}, and suppose moreover that $\supp(\rho_\e) \subset \M_r$ for all $\e>0$. Then \eqref{eq:ls-result} holds.
\end{lem}

\begin{proof}
Note that, given $\e >0$, there exists a sequence $(\nu_\e^n)_n$ such that $\nu_\e^n \in \P(\M)$ for all $n\in \N$, $d_2(\hat\rho_\e,\nu_\e^n) \to 0$ as $n\to \infty$ and 
\be\label{eq:slope-hat}
   |\p E|(\hat \rho_\e) = \lim_{n\to\infty} \frac{\left(E(\hat \rho_\e) - E(\nu_\e^n)\right)_+}{d_2(\hat \rho_\e,\nu_\e^n)}.
\ee
We want to construct $(\mu_\e^n)_n \subset \P_2(\R^d)$ such that $d_2(\rho_\e,\mu_\e^n) \to0$ as $n\to\infty$ and
\be\label{eq:ls-mu-nu}
	\liminf_{\e\to0} \lim_{n\to\infty} \frac{(E_\e(\rho_\e) - E_\e(\mu_\e^n))_+}{d_2(\rho_\e,\mu_\e^n)} \geq \liminf_{\e\to0} \lim_{n\to\infty} \frac{\left(E(\hat \rho_\e) - E(\nu_\e^n)\right)_+}{d_2(\hat \rho_\e,\nu_\e^n)}.
\ee
Indeed, by the definition of $|\p E_\e|(\rho_\e)$, this directly implies the desired result \eqref{eq:ls-result}.

Note that, because $\supp(\rho_\e) \subset \M_r$, we have $p_\e = 1$, $\sigma_\e=\rho_\e$ and $\hat\rho_\e = (\Pi_\M)_\#\rho_\e$.

\medskip
\textit{Step 1: constructing $(\mu_\e^n)_n$.} Fix $\e>0$ and $n\in \N$ in this step. Let us write $\theta_\e^n \in \P(\R^d\times\R^d\times\R^d)$ a plan with $\supp(\theta_\e^n) \subset \M_r \times \M \times \M$ such that 
\be \label{eq:projs-theta}
	(\tilde \pi_{12})_\# \theta_\e^n = (\id,\Pi_\M)_\# \rho_\e \quad \mbox{and} \quad (\tilde \pi_{23})_\# \theta_\e^n = \omega_\e^n,
\ee
where $\omega_\e^n$ is the optimal transport plan between $\hat \rho_\e$ and $\nu_\e^n$, where we recall $\nu_\e^n$ is defined in \eqref{eq:slope-hat}. Here the notation $(\id,\Pi_\M)$ stands for the map defined by $(\id,\Pi_\M)(x) = (x,\Pi_\M(x))$ for all $x\in\M_r$. The existence of the ``compound'' plan $\theta_\e^n$ is justified by \cite[Lemma 5.3.2]{AGS}. In particular, $(\tilde \pi_{12})_\# \theta_\e^n$ is the optimal transport plan between $\rho_\e$ and $\hat \rho_\e$. For all Borel sets $X,Y \subset \R^d$, we then define $\gamma_\e^n \in \P(\R^d\times \R^d)$ by
\be\label{eq:gamma}
	\gamma_\e^n(X \times Y) = \begin{cases} (\tilde \pi_1, \zeta)_\# \theta_\e^n & \mbox{if $X,Y \subset \M_r$}\\ 0 & \mbox{otherwise}, \end{cases}
\ee
where $\zeta \: \R^d \times \M \times \M \to \R^d$ is given by $\zeta(x,p,q) = x + q - p$ for all $(x,p,q) \in \R^d \times \M \times \M$. We now choose
\bes
	\mu_\e^n = (\pi_2)_\# \gamma_\e^n.
\ees
We have $(\pi_1)_\# \gamma_\e^n = \rho_\e$, so that $\gamma_\e^n$ is a transport plan between $\rho_\e$ and $\mu_\e^n$. Figure \ref{fig:projection} illustrates this construction.

\bfig[!ht]
\centering
	\includegraphics[scale=1.25]{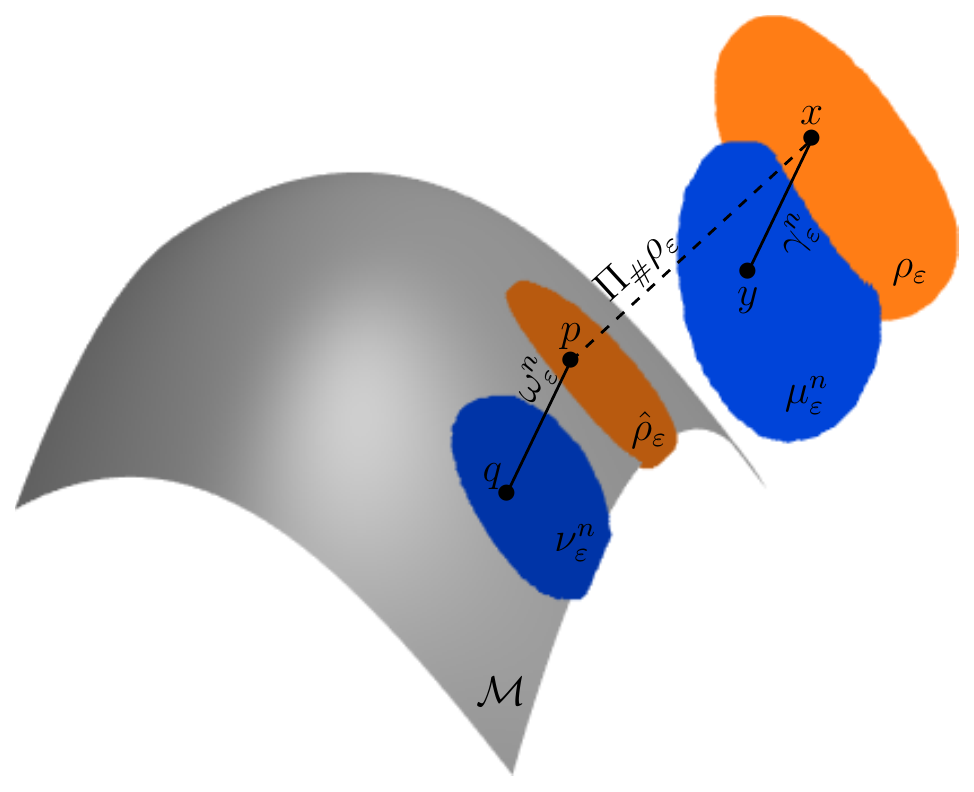}
	\caption{Construction of $\mu_\e^n$\label{fig:projection}.} 
\efig

\medskip
\textit{Step 2: checking $d_2(\rho_\e,\mu_\e^n) \to 0$ as $n\to \infty$.} Let $\e>0$. For all $n\in\N$, because $\gamma_\e^n$ is a transport plan between $\rho_\e$ and $\mu_\e^n$, \eqref{eq:projs-theta} and \eqref{eq:gamma} imply
\begin{align*}
	d_2^2(\rho_\e,\mu_\e^n) &\leq \irdrd |x-y|^2 \d \gamma_\e^n(x,y) = \int_{\M_r\times\M_r} |x-y|^2 \d (\tilde \pi_1,\zeta)_\# \theta_\e^n(x,y)\\
	&=  \int_{\M_r\times \M \times \M} |\tilde \pi_1(x,p,q)-\zeta(x,p,q)|^2 \d \theta_\e^n(x,p,q)\\
	&= \int_{\M_r\times \M \times \M} |x - (x + q - p)|^2 \d \theta_\e^n(x,p,q)\\
	&= \int_{\M_r\times \M \times \M} |p-q|^2 \d \theta_\e^n(x,p,q) = \int_{\M \times \M} |p-q|^2 \d (\tilde \pi_{23})_\# \theta_\e^n(p,q)\\
	&= \int_{\M \times \M} |p-q|^2 \d \omega_\e^n(p,q) = d_2^2(\hat \rho_\e,\nu_\e^n).
\end{align*}
We therefore have, for all $n\in\N$,
\be\label{eq:d2-ineq}
	d_2(\rho_\e,\mu_\e^n) \leq d_2(\hat \rho_\e,\nu_\e^n).
\ee
Since $d_2(\hat \rho_\e,\nu_\e^n) \to 0$ as $n \to \infty$, we have that $d_2(\rho_\e,\mu_\e^n) \to 0$ as well. 

\medskip
\textit{Step 3: getting \eqref{eq:ls-mu-nu}.} Let $n\in\N$. Note that, setting $\Sigma = \M_r\times \M\times\M$ and $d = d_\M$,
\begin{align*}
	\irdrd \left( d(x)^2 - d(y)^2 \right) \d \gamma_\e^n(x,y) &= \int_{\M_r\times\M_r} \left( d(x)^2 - d(y)^2 \right) \d (\tilde \pi_1,\zeta)_\# \theta_\e^n(x,y)\\
	&= \int_{\Sigma} \left( d(x)^2 - d(x+q-p)^2 \right) \d \theta_\e^n(x,p,q)\\
	&= \int_{\Sigma} \left( d(x)^2 - |x+q-p - \Pi_\M(x+q-p)|^2 \right) \d \theta_\e^n(x,p,q)\\
	&\geq \int_{\Sigma} \left( d(x)^2 - |x+q-p - q|^2 \right) \d \theta_\e^n(x,p,q)\\
	&= \int_{\Sigma} \left( d(x)^2 - |x-p|^2 \right) \d \theta_\e^n(x,p,q)\\
	&= \int_{\M_r} d(x)^2 \d\rho_\e(x) - \int_{\M_r\times \M} |x-p|^2 \d (\id,\Pi_\M)_\# \rho_\e(x,p)\\
	&= \int_{\M_r} d(x)^2 \d\rho_\e(x) - \int_{\M_r} |x-\Pi_\M(x)|^2 \d \rho_\e(x)\\
	&= \int_{\M_r} d(x)^2 \d\rho_\e(x) - \int_{\M_r} d(x)^2 \d \rho_\e(x) = 0,
\end{align*}
which leads to
\bes
	E_\e(\rho_\e) - E_\e(\mu_\e^n)\geq \frac{1}{2} \int_{\M_r\times\M_r} \left( W(x-u) - W(y-v) \right) \d \gamma_\e^n(u,v) \d\gamma_\e^n(x,y).
\ees
Using \eqref{eq:gamma} we get
\bes
	E_\e(\rho_\e) - E_\e(\mu_\e^n) \geq \frac{1}{2} \int_{\Sigma} \int_{\Sigma} \left( W(x-u) - W(x+q-p - u-t+s) \right) \d \theta_\e^n(u,s,t) \d\theta_\e^n(x,p,q).
\ees
Because $W$ is of class $C^2$, the integrand above verifies
\bes
	W(x-u) - W(x+q-p - u-t+s) \geq \grad W(x-u) \cdot (p-q - (s-t)) - A (|p-q|^2 + |s-t|^2),
\ees
where $A>0$ is a constant depending on the gradient of $W$. We also have that
\bes
	\grad W(x-u)\cdot (p-q - (s-t)) \geq \grad W(p-s)\cdot (p-q - (s-t)) -B (|x-p| + |u-s|) (|p-q| + |s-t|),
\ees
where $B>0$ is a constant depending on the Hessian of $W$. Thus,
\begin{align*}
	E_\e(\rho_\e) - E_\e(\mu_\e^n) &\geq \frac{1}{2} \int_{\Sigma} \int_{\Sigma} \grad W(p-s) \cdot (p-q - (s-t)) \d \theta_\e^n(u,s,t) \d\theta_\e^n(x,p,q)\\
	&\phantom{{}\geq{}}- \frac{B}{2} \int_{\Sigma} \int_{\Sigma} (|x-p| + |u-s|) (|p-q| + |s-t| ) \d \theta_\e^n(u,s,t) \d\theta_\e^n(x,p,q)\\
	&\phantom{{}\geq{}}- \frac{A}{2} d_2^2(\hat\rho_\e,\nu_\e^n).
\end{align*}
Now, since both $\supp(\hat \rho_\e)$ and $\supp(\nu_\e^n)$ are contained in $\M$,
\begin{align*}
	E(\hat \rho_\e) - E(\nu_\e^n) &= \frac12 \int_{\M \times \M} \int_{\M \times \M} \left( W(p-s) - W(q-t) \right) \d \omega_\e^n(s,t) \d\omega_\e^n(p,q)\\
	&= \frac12 \int_{\Sigma} \int_{\Sigma} \left( W(p-s) - W(q-t) \right) \d \theta_\e^n(u,s,t) \d\theta_\e^n(x,p,q)
\end{align*}
The integrand above satisfies
\bes
	W(p-s) - W(q-t) \leq \grad W(p-s)\cdot (p-s - (q-t)) + C(|p-q|^2 + |s-t|^2),
\ees
where $C>0$ is a constant depending on the gradient of $W$. Thus,
\bes
	E(\hat \rho_\e) - E(\nu_\e^n) \leq \frac12 \int_{\Sigma} \int_{\Sigma} \grad W(p-s)\cdot (p-q - (s-t)) \d \theta_\e^n(u,s,t) \d\theta_\e^n(x,p,q)+ \frac{C}{2} d_2^2(\hat\rho_\e,\nu_\e^n).
\ees
Therefore,
\begin{align*}
	E_\e(\rho_\e) - E_\e(\mu_\e^n) &\geq E(\hat \rho_\e) - E(\nu_\e^n) - D d_2^2(\hat\rho_\e,\nu_\e^n)\\
	&\phantom{{}\geq{}}- \frac{B}{2} \int_{\Sigma} \int_{\Sigma} (|x-p| + |u-s|) (|p-q| + |s-t| ) \d \theta_\e^n(u,s,t) \d\theta_\e^n(x,p,q),
\end{align*}
where $D = \frac12(A + C)$. Furthermore, by the Cauchy--Schwarz inequality and \eqref{eq:projs-theta},
\begin{align*}
	&\int_{\Sigma} \int_{\Sigma} (|x-p| + |u-s|)\cdot (|p-q| + |s-t| ) \d \theta_\e^n(u,s,t) \d\theta_\e^n(x,p,q)\\
	&= \int_{\Sigma} |x-p||p-q| \d\theta_\e^n(x,p,q) + \int_{\Sigma} |u-s||s-t| \d\theta_\e^n(u,s,t)\\
	&\phantom{{}={}}+ \int_{\Sigma} \int_{\Sigma} |x-p||s-t|\d \theta_\e^n(u,s,t) \d\theta_\e^n(x,p,q)\\
	&\phantom{{}={}}+ \int_{\Sigma} \int_{\Sigma} |u-s||p-q|\d \theta_\e^n(u,s,t) \d\theta_\e^n(x,p,q)\\
	&= 2\int_{\Sigma} |x-p||p-q| \d\theta_\e^n(x,p,q)\\
	&\phantom{{}={}}+ 2\int_{\Sigma} \int_{\Sigma} |x-p||s-t|\d \theta_\e^n(u,s,t) \d\theta_\e^n(x,p,q)\\
	&\leq 2 \sqrt{\int_{\Sigma} |x-p|^2 \d\theta_\e^n(x,p,q)} \sqrt{\int_{\Sigma} |p-q|^2 \d\theta_\e^n(x,p,q)} \\
	&\phantom{{}={}}+ 2\sqrt{ \int_{\Sigma} |x-p|^2 \d\theta_\e^n(x,p,q)} \sqrt{ \int_{\Sigma} |s-t|^2 \d \theta_\e^n(u,s,t) } \\
	&= 4 \sqrt{\int_{\M_r \times \M} |x-p|^2 \d (\id,\Pi_\M)_\# \rho_\e(x,p)} \sqrt{\int_{\M \times \M} |p-q|^2 \d\omega_\e^n(p,q)}\\
	&= 4 d_2(\hat\rho_\e,\rho_\e) d_2(\hat\rho_\e,\nu_\e^n).
\end{align*}
Hence
\be\label{eq:E-ineq}
	E_\e(\rho_\e) - E_\e(\mu_\e^n) \geq E(\hat \rho_\e) - E(\nu_\e^n) - 2B d_2(\hat\rho_\e,\rho_\e) d_2(\hat\rho_\e,\nu_\e^n)- D p_\e d_2^2(\hat\rho_\e,\nu_\e^n).
\ee

Finally, using \eqref{eq:d2-ineq} and \eqref{eq:E-ineq} we yield
\begin{align*}
	\lim_{n\to\infty} \frac{(E_\e(\rho_\e) - E_\e(\mu_\e^n))_+}{d_2(\hat\rho_\e,\mu_\e^n)} &\geq \lim_{n\to\infty} \frac{\left( E(\hat \rho_\e) - E(\nu_\e^n) - 2B d_2(\hat\rho_\e,\rho_\e) d_2(\hat\rho_\e,\nu_\e^n)- D d_2^2(\hat\rho_\e,\nu_\e^n) \right)_+}{d_2(\hat \rho_\e,\nu_\e^n)}\\
	&\geq \lim_{n\to\infty} \left( \frac{\left( E(\hat \rho_\e) - E(\nu_\e^n) \right)_+}{d_2(\hat \rho_\e,\nu_\e^n)} - D d_2(\hat\rho_\e,\nu_\e^n) \right) - 2 B d_2(\hat\rho_\e,\rho_\e)\\
	&= \lim_{n\to\infty} \frac{\left( E(\hat \rho_\e) - E(\nu_\e^n) \right)_+}{d_2(\hat \rho_\e,\nu_\e^n)} - 2 B d_2(\hat\rho_\e,\rho_\e).
\end{align*}
Thus, since $d_2(\hat \rho_\e,\rho_\e) \to 0$ as $\e\to0$, we obtain \eqref{eq:ls-mu-nu}.
\end{proof}

We now relax the support assumption of Lemma \ref{lem:main}.

\begin{lem}
	Let $(\rho_\e)_\e$ be as in Theorem \ref{thm:main-epsilon}. Then \eqref{eq:ls-result} holds.
\end{lem}

\begin{proof}
We follow the same strategy as in the proof of Lemma \ref{lem:main}, the only difference being that we have to account for the ``extra'' mass $\rho_\e(\R^d \setminus\M_r)$. It is still enough to prove \eqref{eq:ls-mu-nu}, although now $p_\e \in (0,1]$ and $\hat \rho_\e = (\Pi_\M)_\# \sigma_\e$, where we recall that $\sigma_\e = (\rho_\e\restrict{\M_r})/p_\e$. 

\medskip
\textit{Step 1: constructing $(\mu_\e^n)_n$.} Fix $\e>0$ and $n\in \N$ in this step. Let us write $\theta_\e^n \in \P(\R^d\times\R^d\times\R^d)$ a plan with $\supp(\theta_\e^n) \subset \M_r \times \M \times \M$ such that 
\be \label{eq:projs-theta-1}
	(\tilde \pi_{12})_\# \theta_\e^n = (\id,\Pi_\M)_\# \sigma_\e \quad \mbox{and} \quad (\tilde \pi_{23})_\# \theta_\e^n = \omega_\e^n,
\ee
where $\omega_\e^n$ is the optimal transport plan between $\hat \rho_\e$ and $\nu_\e^n$. The existence of the ``compound'' plan $\theta_\e^n$ is justified by \cite[Lemma 5.3.2]{AGS}. In particular, $(\tilde \pi_{12})_\# \theta_\e^n$ is the optimal transport plan between $\sigma_\e$ and $\hat \rho_\e$. For all Borel sets $X,Y \subset \R^d$, we then define $\beta_\e^n \in \P(\R^d\times \R^d)$ by
\be\label{eq:beta}
	\beta_\e^n(X \times Y) = \begin{cases} (\tilde \pi_1, \zeta)_\# \theta_\e^n & \mbox{if $X,Y \subset \M_r$}\\ 0 & \mbox{otherwise}, \end{cases}
\ee
where $\zeta \: \R^d \times \M \times \M \to \R^d$ is given by $\zeta(x,p,q) = x + q - p$ for all $(x,p,q) \in \R^d \times \M \times \M$.
We now define $\gamma_\e^n \in \P(\R^d\times\R^d)$ by
\be \label{eq:gamma-1}
	\gamma_\e^n = p_\e \beta_\e^n + (\id_2)_\# (\rho_\e\restrict{\R^d\setminus \M_r}),
\ee
and choose
\bes
	\mu_\e^n = (\pi_2)_\# \gamma_\e^n.
\ees 
We have $(\pi_1)_\# \gamma_\e^n = p_\e \sigma_\e + \rho_\e\restrict{\R^d \setminus\M_r} = \rho_\e$, so that $\gamma_\e^n$ is a transport plan between $\rho_\e$ and $\mu_\e^n$. 

Note that the definition of $\gamma_\e^n$ in \eqref{eq:gamma-1} is where the main difference with the proof of Lemma \ref{lem:main} lies. Indeed, to account for the mass outside of $\M_r$ we have added the term $(\id_2)_\# (\rho_\e\restrict{\R^d\setminus \M_r})$, which describes the fact that we do not project, or even move, the mass in $\R^d\setminus\M_r$ since it is negligible as $\e \to0$. 

\medskip
\textit{Step 2: checking $d_2(\rho_\e,\mu_\e^n) \to 0$ as $n\to \infty$.} Let $\e>0$. Very similarly as in Step 2 of the proof of Lemma \ref{lem:main}, we get
\begin{align*}
	d_2^2(\rho_\e,\mu_\e^n) &\leq \irdrd |x-y|^2 \d \gamma_\e^n(x,y)\\
	&= p_\e \irdrd |x-y|^2 \d \beta_\e^n(x,y) + \irdrd |x-y|^2 \d (\id_2)_\# (\rho_\e\restrict{\R^d\setminus \M_r})(x,y)\\
	&= p_\e \int_{\M \times \M} |p-q|^2 \d (\tilde \pi_{23})_\# \theta_\e^n(p,q) = p_\e \int_{\M \times \M} |p-q|^2 \d \omega_\e^n(p,q) = p_\e d_2^2(\hat \rho_\e,\nu_\e^n).
\end{align*}
We therefore have, for all $n\in\N$,
\be\label{eq:d2-ineq-1}
	d_2(\rho_\e,\mu_\e^n) \leq \sqrt{p_\e} d_2(\hat \rho_\e,\nu_\e^n).
\ee
Since $d_2(\hat \rho_\e,\nu_\e^n) \to 0$ as $n \to \infty$, we have that $d_2(\rho_\e,\mu_\e^n) \to 0$ as well. 

\medskip
\textit{Step 3: getting \eqref{eq:ls-mu-nu}.} Let $n\in\N$. Again, setting $\Sigma = \M_r\times \M\times\M$ and $d=d_\M$, and following the proof of Lemma \ref{lem:main}, yields
\begin{align*}
	\frac{1}{p_\e}\irdrd \left( d(x)^2 - d(y)^2 \right) \d \gamma_\e^n(x,y) &= \int_{\M_r\times\M_r} \left( d(x)^2 - d(y)^2 \right) \d \beta_\e^n(x,y)\\
	&= \int_{\M_r} d(x)^2 \d\sigma_\e(x) - \int_{\M_r\times \M} |x-p|^2 \d (\id,\Pi_\M)_\# \sigma_\e(x,p)\\
	&= \int_{\M_r} d(x)^2 \d\sigma_\e(x) - \int_{\M_r} |x-\Pi_\M(x)|^2 \d \sigma_\e(x)\\
	&= \int_{\M_r} d(x)^2 \d\sigma_\e(x) - \int_{\M_r} d(x)^2 \d \sigma_\e(x) = 0,
\end{align*}
which leads to
\bes
	E_\e(\rho_\e) - E_\e(\mu_\e^n)\geq \frac{1}{2} \int_{\M_r\times\M_r} \left( W(x-u) - W(y-v) \right) \d \gamma_\e^n(u,v) \d\gamma_\e^n(x,y),
\ees
that is, by \eqref{eq:gamma-1},
\bes
	E_\e(\rho_\e) - E_\e(\mu_\e^n)\geq \frac{p_\e^2}{2} \int_{\M_r\times\M_r} \left( W(x-u) - W(y-v) \right) \d \beta_\e^n(u,v) \d\beta_\e^n(x,y)
\ees
Using \eqref{eq:beta} we get
\bes
	E_\e(\rho_\e) - E_\e(\mu_\e^n) \geq \frac{p_\e^2}{2} \int_{\Sigma} \int_{\Sigma} \left( W(x-u) - W(x+q-p - u-t+s) \right) \d \theta_\e^n(u,s,t) \d\theta_\e^n(x,p,q)
\ees
Still following the proof of Lemma \ref{lem:main}, we obtain
\begin{align*}
	E_\e(\rho_\e) - E_\e(\mu_\e^n) &\geq p_\e^2 \left( E(\hat \rho_\e) - E(\nu_\e^n) \right) - D p_\e^2 d_2^2(\hat\rho_\e,\nu_\e^n)\\
	&\phantom{{}\geq{}}- \frac{Bp_\e^2}{2} \int_{\Sigma} \int_{\Sigma} (|x-p| + |u-s|) (|p-q| + |s-t| ) \d \theta_\e^n(u,s,t) \d\theta_\e^n(x,p,q),
\end{align*}
and 
\be\label{eq:E-ineq-1}
	E_\e(\rho_\e) - E_\e(\mu_\e^n) \geq p_\e^2 \left( E(\hat \rho_\e) - E(\nu_\e^n) \right) - 2B p_\e^2 d_2(\hat\rho_\e,\rho_\e) d_2(\hat\rho_\e,\nu_\e^n)- D p_\e^2 d_2^2(\hat\rho_\e,\nu_\e^n),
\ee
where the constants $B$ and $D$ are defined as in the proof of Lemma \ref{lem:main}.

Finally, using \eqref{eq:d2-ineq-1} and \eqref{eq:E-ineq-1} we yield
\begin{align*}
	&\lim_{n\to\infty} \frac{(E_\e(\rho_\e) - E_\e(\mu_\e^n))_+}{d_2(\hat\rho_\e,\mu_\e^n)}\\
	&\geq \lim_{n\to\infty} \frac{\left( p_\e^2 \left( E(\hat \rho_\e) - E(\nu_\e^n) \right) - 2B p_\e^2 d_2(\hat\rho_\e,\rho_\e) d_2(\hat\rho_\e,\nu_\e^n)- Dp_\e^2 d_2^2(\hat\rho_\e,\nu_\e^n) \right)_+}{\sqrt{p_\e} d_2(\hat \rho_\e,\nu_\e^n)}\\
	&\geq p_\e^{3/2} \lim_{n\to\infty} \left( \frac{\left( E(\hat \rho_\e) - E(\nu_\e^n) \right)_+}{d_2(\hat \rho_\e,\nu_\e^n)} - D d_2(\hat\rho_\e,\nu_\e^n) \right) - 2 Bp_\e^{3/2} d_2(\hat\rho_\e,\rho_\e)\\
	&= p_\e^{3/2} \lim_{n\to\infty} \frac{\left( E(\hat \rho_\e) - E(\nu_\e^n) \right)_+}{d_2(\hat \rho_\e,\nu_\e^n)} - 2 Bp_\e^{3/2} d_2(\hat\rho_\e,\rho_\e).
\end{align*}
Thus, by $p_\e \to1$ and $d_2(\hat \rho_\e,\rho_\e) \to 0$ as $\e\to0$, we obtain \eqref{eq:ls-mu-nu}, which ends the proof.
\end{proof}

\begin{rem}
	The convergence given in Theorem \ref{thm:main-epsilon} and the second line in \eqref{eq:serfaty-conseq} imply
\bes
	\frac12 \ird W*\rho_\e(t,x) \d\rho_\e(t,x) + \frac1\e \ird d_\M(x)^2 \d \rho_\e(t,x) \xrightarrow[\e\to0]{} \frac12 \int_\M W*\rho(t,x) \d \rho(t,x).
\ees
By continuity and boundedness from below of $W$, $\ird W*\rho_\e(t,x) \d\rho_\e(t,x) \to \int_\M W*\rho(t,x) \d \rho(t,x)$ as $\e\to0$, and thus
\bes
	\ird d_\M(x)^2 \d \rho_\e(t,x) = o(\e) \quad \mbox{as $\e\to0$},
\ees
which gives the ``rate of attraction'' to $\M$ as $\e\to0$.
\end{rem}

\subsection{Stability}

Theorem \ref{thm:main-epsilon} shows, by approximation via the gradient flow for $E_\e$, that the gradient flow for $E$ has a solution. It does not, however, prove uniqueness of such a solution. It turns out that a Wasserstein stability estimate holds on $\M$, as shown by Proposition \ref{prop:stability} below. This ensures the uniqueness of the solution to the gradient flow for $E$, and therefore yields the well-posedness of the Wasserstein gradient flow for $E$.

We omit the proof of the stability result as it follows the exact same steps as those found in the proof of \cite[Proposition 3.1 and Theorem 1.6]{CSW} with the additional help of the lemma below whose proof can be found in \cite[Proposition 3.1]{RZ}. Before stating the lemma and stability result, we recall the notion of proximal normal cone of $\M$ at a point $x\in\M$: the \emph{proximal normal cone} of $\M$ at $x$ is the set $N_x\M$ given by
\bes
N_x\M = \{ v\in\R^d \st \exists\, \alpha\in(0,\infty),\, d_\M(x+\alpha v) = \alpha |v| \},
\ees
that is, $N_x\M$ is the set of vectors $v\in\R^d$ so that there exists $\alpha>0$ such that $x$ is a closest point to $x+\alpha v$ on $\M$. 

\begin{lem}\label{lem:positive-reach-ineq}
  Let $x,y \in \M$. For every $v\in N_x\M$ there holds
  \bes
  \ap{y-x,v} \leq \frac{|y-x|^2|v|}{2\eta_\M},
  \ees
where we recall $\eta_\M>0$ is the reach of $\M$.
\end{lem}

\begin{prop}[stability estimate] \label{prop:stability}
	Denote by $\lambda_W\leq0$ a semiconvexity constant of $W$. Let $\rho_1$ and $\rho_2$ be two gradient flows for $E$ (which we know exist by Theorem \ref{thm:main-epsilon}) starting from $\rho_1^0 \in \P(\M)$ and $\rho_2^0 \in \P(\M)$, respectively. Then, for all $t\geq0$,
	\bes
		d_2(\rho_1(t),\rho_2(t)) \leq e^{\left(-\lambda_W + \frac{\|\grad W\|_{L^\infty(\M)}}{\eta_\M} \right) t} d_2(\rho_1^0,\rho_2^0).
	\ees
\end{prop}


\section{Projected gradient flow scheme (proof of Theorem \ref{thm:main-projected})}\label{sec:result-projected}

We consider $\mu$ and $\rho$ solutions to the classical full-space nonlocal-interaction equation and the gradient flow for $E$, respectively, that is,
\be\label{eq:interaction-equations}
	\begin{cases} \p_t \mu + \dive(\mu v) = 0,\\ v = -\grad W * \mu, \end{cases} \quad \mbox{and} \quad \begin{cases} \p_t \rho + \dive(\rho u) = 0,\\ u = P_\M(-\grad W * \rho). \end{cases}
\ee
Note that for any $x\in\M$, the projections on $T_x\M$ and $T_x^\mathrm{in}\M$ satisfy $\Pi_{T_x\M}(v) - \Pi_{T_x^\mathrm{in}\M}(v) \in N_x\M$ for all $v\in\R^d$, where we recall that $N_x\M$ is the proximal normal cone of $\M$ at $x$. To simplify the notation in this section, we shall use $P_x$ for $\Pi_{T_x^\mathrm{in}\M}$ and $P_x^\mathrm{T}$ for $\Pi_{T_x\M}$.


For all $t\geq0$, let $\Phi_t\: \R^d \to \R^d$ be the flow map associated to the classical full-space nonlocal-interaction equation and let $\Psi_t\: \M \to \M$ be that associated to the gradient flow for $E$. Then, the solutions to \eqref{eq:interaction-equations} starting from some $\rho^0\in\P(\M)$ are given by
\bes
	\mu(t) = (\Phi_t)_\#\rho^0 \quad \mbox{and} \quad \rho(t) = (\Psi_t)_\#\rho^0.
\ees
Fixing $r\in(0,\eta_\M)$ we write $\M_r$ the $r$-neighborhood of $\M$, i.e., $\M_r = \{ x\in \R^d \st d_\M(x) < r\}$, so that the projection $\Pi_\M\:\M_r\to\M$ on $\M$ is well-defined. Let us take a time step size $\tau>0$ small enough such that $\supp((\Phi_t)_\#\nu) \subset \M_r$ for all $\nu\in\P(\M)$ and $t\in[0,\tau]$. Such a time step exists since $W \in C^2(\R^d)$ and $\M$ is compact; indeed, this ensures that their exists a constant $C>0$ such that for all $\nu\in\P(\M)$ there holds $\norm{\grad W * \nu}_{L^\infty(\M_r)} \leq \norm{\grad W}_{L^\infty(\M_r)} \leq C$. We can then define the sequence $(\nu_n^\tau)_n \subset \P(\M)$ as follows: for any integer $n\geq0$,
\bes
	\begin{cases} \nu_0^\tau = \rho^0,\\
	\nu_{n+1}^\tau = (\Pi_\M)_\# ((\Phi_\tau)_\#\nu_n^\tau). \end{cases}
\ees
We define the interpolation $\nu^\tau \: [0,\infty) \to \P(\M)$ by
\bes
	\begin{cases} \nu^\tau(0) = \rho^0,\\
	\nu^\tau(t) = (\Pi_\M)_\# ((\Phi_{t-n\tau})_\# \nu_n^\tau) & \mbox{for all $t\in(n\tau,(n+1)\tau]$}. \end{cases}
\ees
In particular, we have
\bes
	\nu^\tau(t) = (\Pi_\M)_\# ((\Phi_t)_\# \rho^0) = (\Pi_\M)_\# \mu(t)  \quad \mbox{for all $t\in[0,\tau]$}.
\ees
We now show that $\nu^\tau$ is a good approximation of $\rho$, that is, $d_2(\nu^\tau(t),\rho(t)) \to 0$ as $\tau\to0$ for all $t\geq0$, which indeed provides the proof of Theorem \ref{thm:main-projected}.

Let $R>0$ be such that $\M_r - \M_r := \{ x-y \st x,y \in \M_r\} \subset B_R$, where $B_R$ is the open ball centered at $0$ with radius $R$. Let $M_\mt{v} = \| \nabla W \|_{L^\infty(B_R)}$. In all of the arguments that follow we further restrict $\tau$ such that 
\[ \tau \leq \frac{\eta_\M}{8 M_\mt{v}}. \]
Writing $h\:[0,\tau] \to \R$ the function defined by $h(t) = d_2(\nu^\tau(t),\rho(t))$ for all $t\in [0,\tau]$, we seek, for a fixed $t\in[0,\tau]$, to estimate 
\be\label{eq:tri-ineq-scheme}
	h^2(t):= d_2^2(\nu^\tau(t),\rho(t)) = \ird | \Pi_\M(\Phi_t(x)) - \Psi_t(x) |^2 \d \rho^0(x).
\ee
Let us take the right-derivative of $h^2$. For any $x\in\R^d$ we write $y = \Phi_t(x)$ and $z = \Psi_t(x)$. Then,
\begin{align*}
\frac12 \frac{\der^+ h^2}{\der t}(t) & =  \ird (\Pi_\M(y) - z ) \left( P_{\Pi_\M(y)}(-\nabla W * \mu_t(y)) - P_{z }(-\nabla W * \rho_t(z))  \right) \d \rho^0(x)\\
	&\leq \left|  \ird (\Pi_\M(y) - z ) \left( P_{\Pi_\M(y)}(-\nabla W * \mu_t(y))  - P^\mathrm{T}_{\Pi_\M(y)}(-\nabla W * \mu_t(y))   \right) \d\rho^0(x)  \right| \\
	& \quad +  \left|  \ird (\Pi_\M(y) - z ) \left( P^\mathrm{T}_{\Pi_\M(y)}(-\nabla W * \mu_t(y))   - P^\mathrm{T}_{z }(-\nabla W * \rho_t(z))   \right)  \d\rho^0(x) \right|  	\\
	& \quad +  \left|  \ird (\Pi_\M(y) - z ) \left(P^\mathrm{T}_{z }(-\nabla W * \rho_t(z))    - P_{z }(-\nabla W * \rho_t(z))   \right)  \d\rho^0(x) \right|,\\
\shortintertext{using Lemma \ref{lem:positive-reach-ineq} to estimate the first and third terms we continue the computation:} 
	& \leq \frac{1}{2 \eta_\M} \ird |\Pi_\M(y) - z |^2 (|\nabla W * \mu_t(y)|+|\nabla W * \rho_t(z)| ) \,  \d\rho^0(x) \\
	& \phantom{\leq} + \left|  \ird (\Pi_\M(y) - z ) \left( P^\mathrm{T}_{\Pi_\M(y)}(-\nabla W * \mu_t(y))   - P^\mathrm{T}_{\Pi_\M(y) }(-\nabla W * \rho_t(z))   \right)  \d\rho^0(x) \right|  \\
	& \phantom{\leq} + \left|  \ird (\Pi_\M(y) - z ) \left(  P^\mathrm{T}_{\Pi_\M(y) }(-\nabla W * \rho_t(z))  - P^\mathrm{T}_{z }(-\nabla W * \rho_t(z))   \right)  \d\rho^0(x) \right|,  \\
\shortintertext{using Proposition 6.2 in \cite{NiSmWe08} to estimate the third term we further have:} 
	& \leq \frac{ \|\nabla W\|_{L^\infty(B_R)} }{\eta_\M} \ird |\Pi_\M(y) - z |^2  \d\rho^0(x)  \\
	& \quad + \ird |\Pi_\M(y) - z| \, \left| \ird \nabla W(\Phi_t(x) - \Phi_t(s)) - \nabla W(\Psi_t(x) - \Psi_t(s)) \d\rho^0(s)   \right| \d\rho^0(x)  \\
	& \quad + \|\nabla W\|_{L^\infty(B_R)} \ird |\Pi_\M(y) - z| \,  \frac{\sqrt 2}{\eta_\M}  d_\M(\Pi_\M(y),  z)  \, \d\rho^0(x),  \\
\shortintertext{using \cite[Proposition 2]{GGHS} and $|\Pi_\M(y) - z| \leq |\Pi_\M(y) - y| + |y-z| \leq 2M_\mt{v} \tau \leq \frac12 \eta_\M$ we finally get:} 
	& \leq \frac{ \|\nabla W\|_{L^\infty(B_R)} }{\eta_\M} \, h^2(t) \\
	& \quad + \|D^2 W\|_{L^\infty(B_R)} \!\!\ird \!|\Pi_\M(y) - z| \!\left| \ird \!|\Phi_t(x) - \Psi_t(x)| \!-\! |\Phi_t(s) - \Psi_t(s)| \d\rho^0(s)   \right| \!\d\rho^0(x)  \\
	& \quad + \frac{5}{\eta_\M} \|\nabla W\|_{L^\infty(B_R)} \ird |\Pi_\M(y) - z|^2  \d\rho^0(x)  \\
	& \leq \frac{ 6 \|\nabla W\|_{L^\infty(B_R)} }{\eta_\M} \, h^2(t)  + 4 \|D^2 W\|_{L^\infty(B_R)} \|\nabla W\|_{L^\infty(B_R)}    \, h(t) t .\\
\end{align*}
Therefore,
\[  \frac{\der^+h}{\der t}(t)  \leq \frac{ 6 \|\nabla W\|_{L^\infty(B_R)} }{\eta_\M} \, h(t)  + 4 \|D^2 W\|_{L^\infty(B_R)}
 \|\nabla W\|_{L^\infty(B_R)}    \, t =: a h(t) + bt, \]
and thus
\[  \frac{\der ^+}{\der t} (h(t)e^{-at}) \leq bt  e^{-at} \leq bt . \]
 Using that $h(0)=0$ we obtain that for all $t \in [0, \tau]$
 \begin{equation} \label{eq:schstab}
d_2(\nu^\tau(t),\rho(t))  =  h(t) \leq \frac12 bt^2 e^{at} = 2\|D^2 W\|_{L^\infty(B_R)}  \|\nabla W\|_{L^\infty(B_R)} t^2 e^{\frac{6\|\grad W\|_{L^\infty(\M)}t}{\eta_\M}},
\end{equation}
confirming that the local error of the proposed scheme is what one would expect for a first-order scheme.

Combining this local error estimate of the scheme with the stability estimates of Lemma  \ref{prop:stability} we obtain, for all $t \geq 0$,
\begin{equation}
 d_2(\nu^\tau(t),\rho(t))  \leq \tau e^{\frac{6\|\grad W\|_{L^\infty(\M)}\tau}{\eta_\M}}\alpha_t,
\end{equation}
where
\begin{equation*}
  \alpha_t = \frac{2\|D^2 W\|_{L^\infty(B_R)}  \|\nabla W\|_{L^\infty(B_R)}}{\left(-\lambda_W + \frac{\|\grad W\|_{L^\infty(\M)}}{\eta_\M} \right) } \left( e^{\left(-\lambda_W + \frac{\|\grad W\|_{L^\infty(\M)}}{\eta_\M} \right) t} -1 \right).
\end{equation*}

\section{Numerics} \label{sec:exp}

We present in this section some experiments illustrating the dynamics of particles following the $\e$-gradient flow and projected gradient flow schemes, whose convergence results have been given in Sections \ref{sec:result-epsilon} and \ref{sec:result-projected}, respectively. We consider here very simple interaction potentials and highlight how the geometry of the domain can have decisive influence on the dynamics. For the code we used, we refer the reader to our \href{https://github.com/francesco-patacchini/interaction-equation-attracting-manifolds}{GitHub} repository \cite{GitHub}. Before discussing the experiments, we present the numerical setting.

\subsection{Numerical schemes}

Our discretization is based on the fact that the gradient flow solutions of \eqref{eq:main} and \eqref{eq:main-epsilon} for initial data which are discrete measures become systems of ODEs. 
On the other hand, one can approximate in Wasserstein (as well as  $\infty$-Wasserstein) distance any desired initial measure by discrete measures. The stability estimate of Proposition \ref{prop:stability} ensures that the particle approximations to \eqref{eq:main} approximate well the solutions over time. 

We consider particle solutions with $N$ particles, $x_i\:[0,\infty) \to \R^d$, $i\in\{1,\dots,N\}$. All particles have same mass $\frac{1}{N}$. The initial condition  $(x_1^0,\dots,x_N^0)$ belongs to $(\R^{d})^N$. We moreover write 
\bes
    \mu_N^0 = \frac1N \sum_{i=1}^N \delta_{x_i^0} \quad \mbox{and} \quad \mu_N(t) = \frac1N \sum_{i=1}^N \delta_{x_i(t)},
\ees
for 
the empirical measures associated to the particles. 

\subsubsection{$\e$-gradient flow scheme} \label{sec:epsilon-scheme}
Recalling that the underlying energy $E_\e$ is given in \eqref{eq:energy-epsilon} and plugging the empirical measures $\mu_N$ into it, we yield the following discrete energy:
\bes
    E_{\e,N}(x_1,\dots,x_N) = \frac{1}{2N^2} \sum_{i=1}^N\sum_{j=1}^N W(x_i-x_j) + \frac{1}{N\e} \sum_{i=1}^N d_\M(x_i)^2.
\ees
The gradient flow \eqref{eq:main-epsilon} reduces to the ODE system, for all $i\in\{1,\dots,N\}$,
\be\label{eq:epsilon-ODE}
    \begin{cases} x_i'(t) = -N\grad_i E_{\e,N}(x_1(t),\dots,x_N(t)),\\ x_i(0) = x_i^0. \end{cases}
\ee
The fact that this ODE system converges to the gradient flow for $E_\e$ as $N\to\infty$ is well-known; see for instance \cite[Theorem 3.1]{CCH}. We finally discretize \eqref{eq:epsilon-ODE} in time via a forward Euler scheme: take a time step size $\tau > 0$ and for all $n\in\N_0$ denote by $x_i^n$ the approximation of $x_i(t)$ for all $t\in (n\tau, (n+1)\tau]$ and apply
\be\label{eq:epsilon-discrete}
x_i^{n+1} = x_i^n - \tau N\grad_i E_{\e,N}(x_1^n,\dots,x_N^n).
\ee
As a stopping criterion for our simulations we either stop once a fixed final time is reached or stop as soon as at time step $n+1$ we find
\be\label{eq:discretized-tol}
|\grad E_{\e,N}(x_1^{n+1},\dots,x_N^{n+1})| < \mt{tol},
\ee
for a tolerance $\mt{tol} > 0$.

\subsubsection{Projected gradient flow scheme} \label{sec:proj-scheme}
The classical interaction energy
\bes
    E(\rho) = \frac12 \ird\ird W(x-y) \d\rho(y)\d\rho(x) \quad \text{for all $\rho\in\P(\R^d)$},
\ees
for $\rho = \mu_N$ has the form
\bes
    E_N(x_1,\dots,x_N) = \frac{1}{2N^2} \sum_{i=1}^N\sum_{j=1}^N W(x_i-x_j). 
\ees
The resulting discretization we choose is the following "splitting" scheme: for all $i\in\{1,\dots,N\}$,
\be\label{eq:projected-ODE}
    \begin{cases} y_i'(t) = -N\grad_i E_N(x_1(t),\dots,x_N(t)),\\
    x_i(t) = \Pi_\M(y_i(t)),\\ 
    x_i(0) = x_i^0, \end{cases}
\ee
where we recall that $\Pi_\M$ is the projection on $\M$ from an $r$-neighborhood of $\M$ for some $r<\eta_\M$. Again, we discretize \eqref{eq:projected-ODE} in time via a forward Euler scheme as described above, which for all $n\in\N_0$ yields
\be\label{eq:projected-discrete}
\begin{cases}
    y_i^{n+1} = y_i^n - \tau N\grad_i E_N(x_1^n,\dots,x_N^n),\\
    x_i^{n+1} = \Pi_\M(y_i^{n+1}).
\end{cases}
\ee
Because in this case there is no reason to think that the gradient of $E_N$ should achieve $0$ asymptotically in the dynamics given by \eqref{eq:projected-discrete}, we use an alternative stopping criterion to that in \eqref{eq:discretized-tol}. Instead, we stop our simulation as soon as at time step $n+1$ we find
\be\label{eq:discretized-tol-proj}
    \frac{|E_{\e,N}(x_1^{n+1},\dots,x_N^{n+1}) - E_{\e,N}(x_1^n,\dots,x_N^n)|}{\tau} < \mt{tol}.
\ee

For both the $\e$-gradient flow and projected gradient flow schemes, whenever we do not know an explicit formula for the distance $d_\M(x)$ of a point $x\in\R^d$ to $\M$, we numerically approximate it by initially sampling the boundary of the set $\M$, linearly interpolating between these sampling points, and then computing the distance from the point $x$ and this linear interpolation of the boundary; from this we also deduce an approximation of the projection $\Pi_\M(x)$. Also, again for both the $\e$-gradient flow and projected gradient flow schemes, for each simulation we choose our initial time step size via a backtracking linesearch and then keep it constant through the rest of the simulation. Other approaches can of course be easily tested, for example using an adaptive time step size.

\subsection{Experiments}

Here we report on several experiments we conducted in one and two dimensions. We generally observed that the projected scheme (cf. Section \ref{sec:proj-scheme} is more robust and converges faster than the $\e$-scheme (cf. Section \ref{sec:epsilon-scheme}). For this reason we shall mostly focus on the former in the following simulations.

\subsubsection{1D domain} We consider the gradient flow \eqref{eq:main-epsilon} approximated by \eqref{eq:epsilon-discrete}, where the domain is the union of an interval and a point: $\M = [-1,1] \cup \{1.5\}$. The initial data are $N=100$ particles arranged as a random sample of the interval $[-1.75,1.75]$. The potential is the attractive potential $W(x) = x^2$ and the attraction parameter is $\e=0.1$. Our tolerance for convergence is $\mt{tol} = 10^{-9}$ (cf. \eqref{eq:discretized-tol}). The dynamics shown on Figure \ref{fig:oneDdisconnected} displays two time scales. On the time scale $\frac{1}{\e}=10$ the particles converge to within roughly $\e$ of the domain. After that, the dynamics first gathers the points in the interval $[-1,1]$ together and then gets them as close to the mass near $x=1.5$ as the domain penalty $\tfrac1\e$ allows.
\bfig[!ht]
\centering
	\includegraphics[scale=0.6]{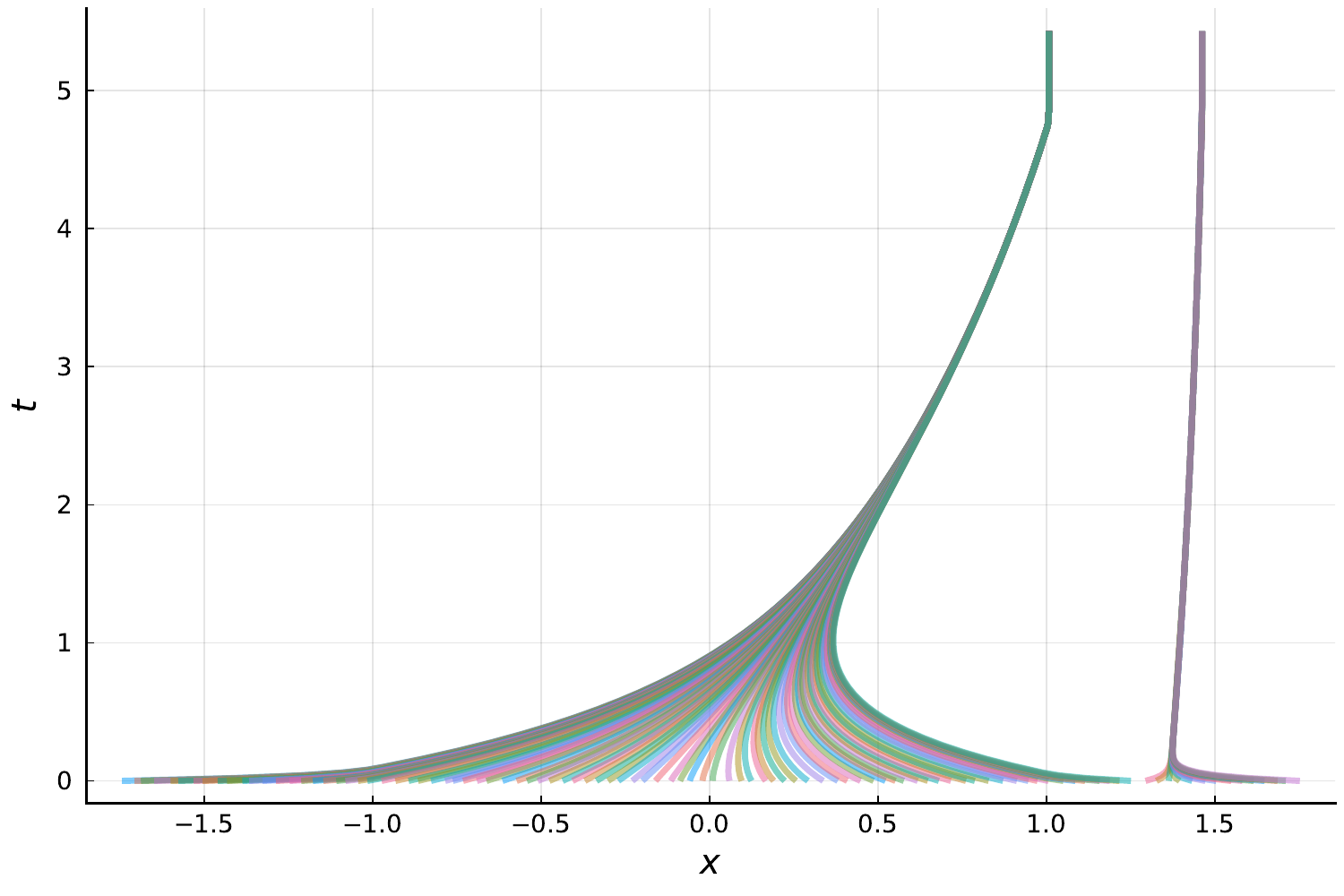}
	\caption{Dynamics of \eqref{eq:main-epsilon} approximated by \eqref{eq:epsilon-discrete} with domain $\M = [-1,1] \cup \{1.5\}$ for an attractive potential. \label{fig:oneDdisconnected}} 
\efig



\subsubsection{Disc in 2D}

We now approximate the gradient flow \eqref{eq:main} using the discretization \eqref{eq:projected-discrete}. The domain is $\M = \overline B(0,1)$, whose boundary is represented in light gray in Figure \ref{fig:disc}. We consider two different potentials: $W(x) = \frac{1}{1+10 |x|^2}$ and $W(x) = \frac{1}{1+|x|^2}$, the first decaying on a shorter length scale than the second. In both cases we initialize the dynamics with $N=196$ points placed on a random perturbation of a uniform rectangular grid within the domain; see blue dots on Figure \ref{fig:disc}. 
The states after time $t=200$ are shown as red dots on Figure \ref{fig:disc}. Two time scales are again observed. Initially the majority of points converges quickly to the boundary. On the slower time scale the points on the boundary migrate to assume nearly uniform distribution. The slow time scale of the motion on the boundary currently prevents us from saying if the final configuration on the right has a uniform distribution on the boundary.

\begin{figure}[!ht]
\subfloat[shorter range potential $W(x) = \frac{1}{1+10 |x|^2}$.]{	\includegraphics[scale=0.51]{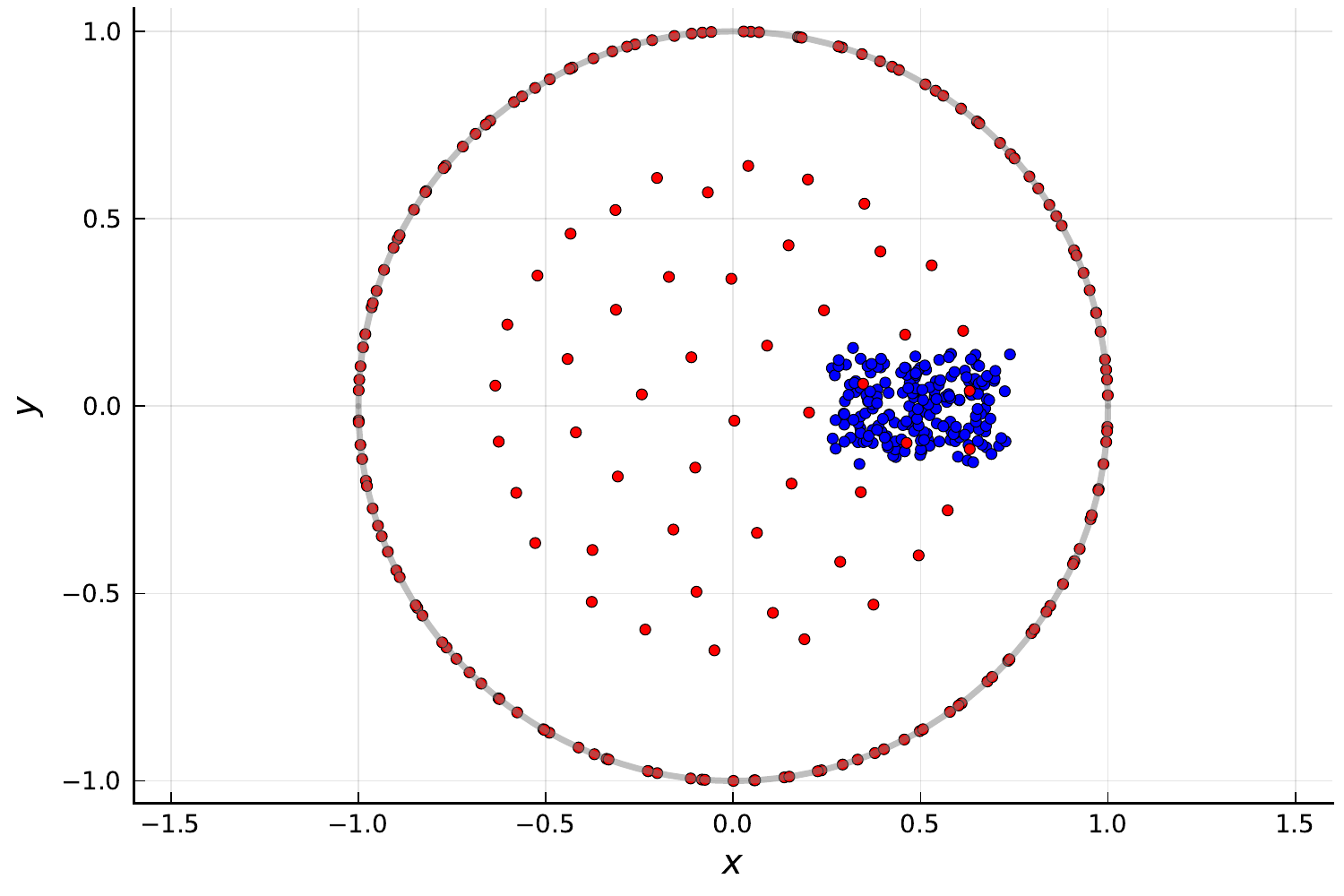}}
      \subfloat[longer range potential $W(x) = \frac{1}{1+ |x|^2}$.]{	\includegraphics[scale=0.51]{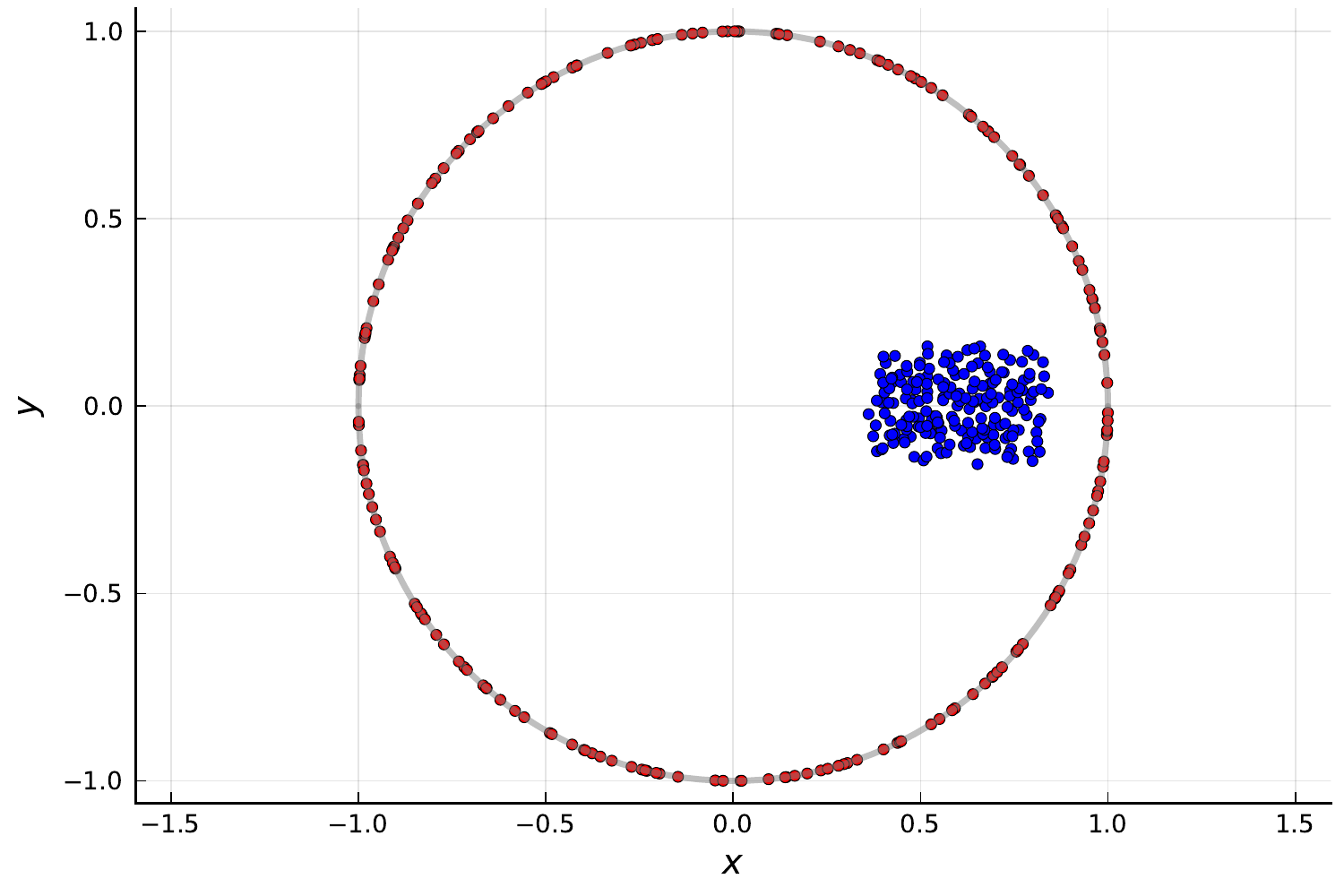}}
       \caption{Dynamics of \eqref{eq:main} approximated by \eqref{eq:projected-discrete} with domain $\M = \overline B(0,1)$ for repulsive potentials with varying length scales. \label{fig:disc}}
       
\end{figure}

\subsubsection{Bean with repulsive potential}

Here we consider the dynamics given by \eqref{eq:projected-discrete} on a bean-shaped domain with smooth boundary depicted in light gray in Figure \ref{fig:bean_interior}. The parametrization of the boundary is in fact given by
\begin{equation*}
  p(x) = \pm 0.4 \sqrt{1 - x^2} (1.1 - \cos(3x)) \quad \text{for all $x\in [-1,1]$}.
\end{equation*}
We consider the repulsive potential $W(x) = \frac{1}{1+|x|^2}$ with distinct randomly perturbed uniform rectangular grids of $N=196$ points as initial distributions. If needed, these grids are projected to our manifold to ensure that the initial data are within our domain. Our tolerance for convergence is taken to be $\mt{tol} = 2\cdot 10^{-10}$ (cf. \eqref{eq:discretized-tol-proj}). While the geometries of the final states, shown as red dots on Figure \ref{fig:bean_interior}, are not that dissimilar, the masses of the points are considerably different, resulting in significantly different energies. We conclude that the configuration on the right is a local minimizer. We noticed that the dynamics from varying initial states would often converge to different local minimizers. Let us clarify that by local minimizers we mean that there are no lower energy states with respect to local perturbations of points. Furthermore we believe that these are also $\infty$-Wasserstein local minimizers. 
In our view the asymmetry of the domain has similar effect to introducing energy barriers and increases the complexity of the energy landscape. It is an interesting theoretical question to understand and predict the features of the energy landscape based on the geometry of the domain, which we shall leave to further investigation.

\begin{figure}[!ht]
\subfloat[symmetric configuration with final energy $E_{196} = 0.260$.]{	\includegraphics[scale=0.52]{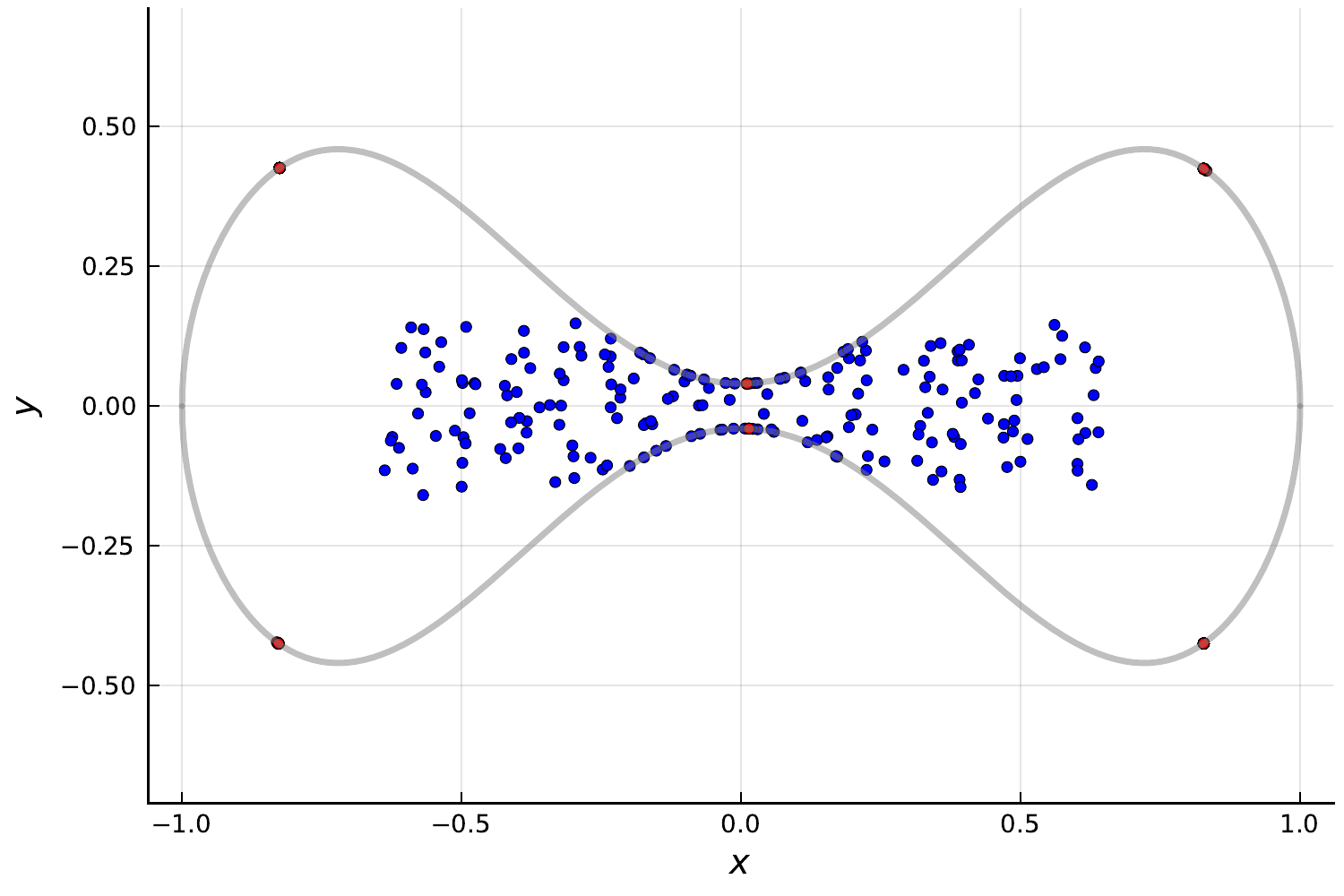}}
      \subfloat[asymmetric configuration with final energy $E_{196} = 0.273$.]{\includegraphics[scale=0.52]{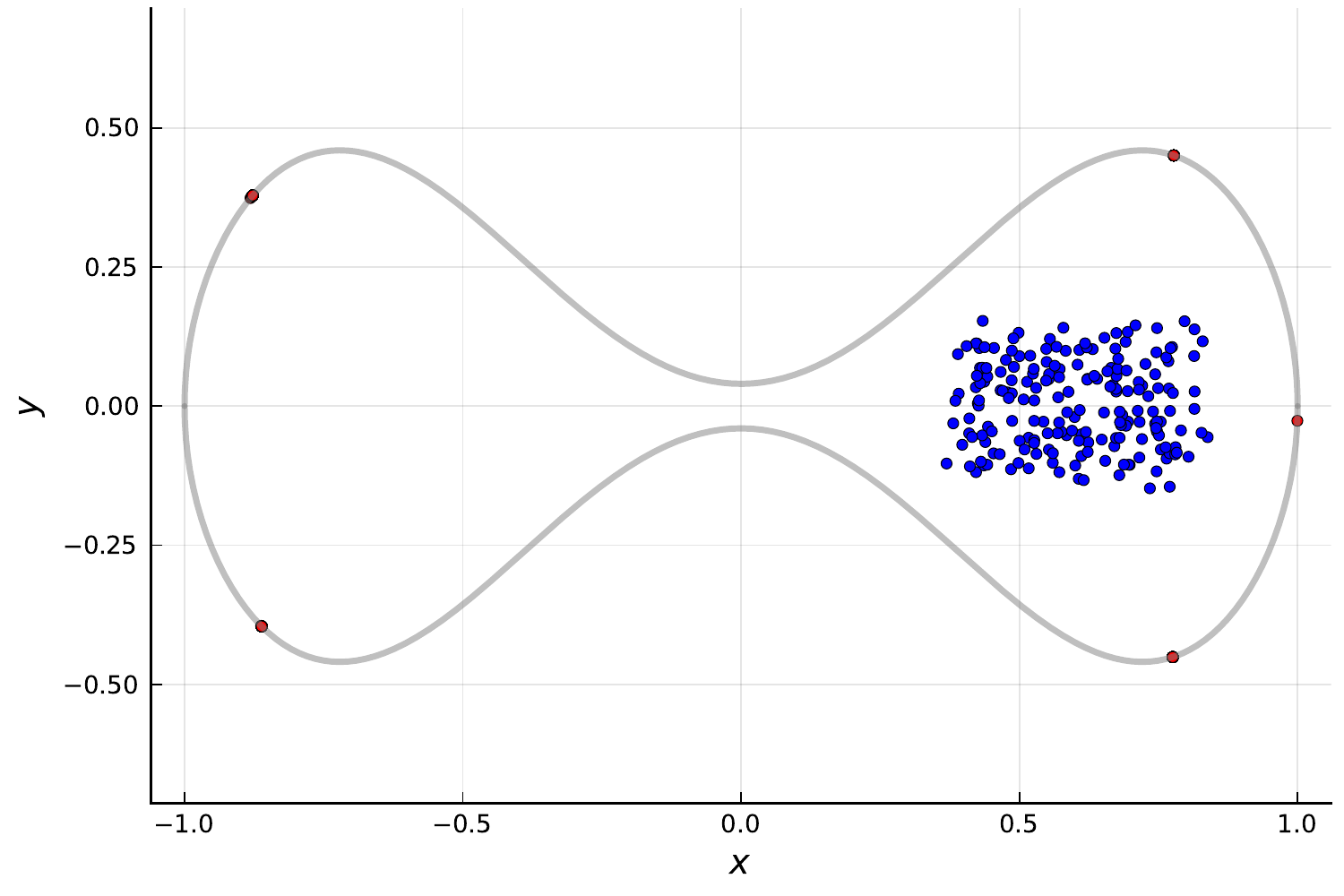}}
       \caption{Dynamics of \eqref{eq:main} approximated by \eqref{eq:projected-discrete} with a bean-shaped domain for a repulsive potential.}
       \label{fig:bean_interior}
\end{figure}



\subsubsection{Bean boundary with attractive potential}

In our final example we consider the domain to be the boundary of the bean-shaped domain from the previous example, on which the dynamics is again given by \eqref{eq:projected-discrete}. Here the initial particles are obtained by projecting a randomly perturbed uniform rectangular grid to the boundary of the bean. We consider the attractive potential $W(x) = -\frac{1}{1+|x|^2}$ and stop the simulation at $t=28$.
While the global minimizer of this energy is achieved when all of the mass is concentrated at a single point, we observe on Figure \ref{fig:bean_bdry} that, due to the nonconvexity of the shape, the dynamics has local minimizers. Such states have a very small basin of attraction, resulting in a metastable-like behavior of the dynamics as illustrated on the right picture of Figure \ref{fig:bean_bdry}. Moreover, we observed that there is a likely local minimizer with three masses (one being on the right-most point of the domain).

\begin{figure}[!ht]
\subfloat[initial (blue dots) and final (red dots) particle positions.]{	\includegraphics[scale=0.51]{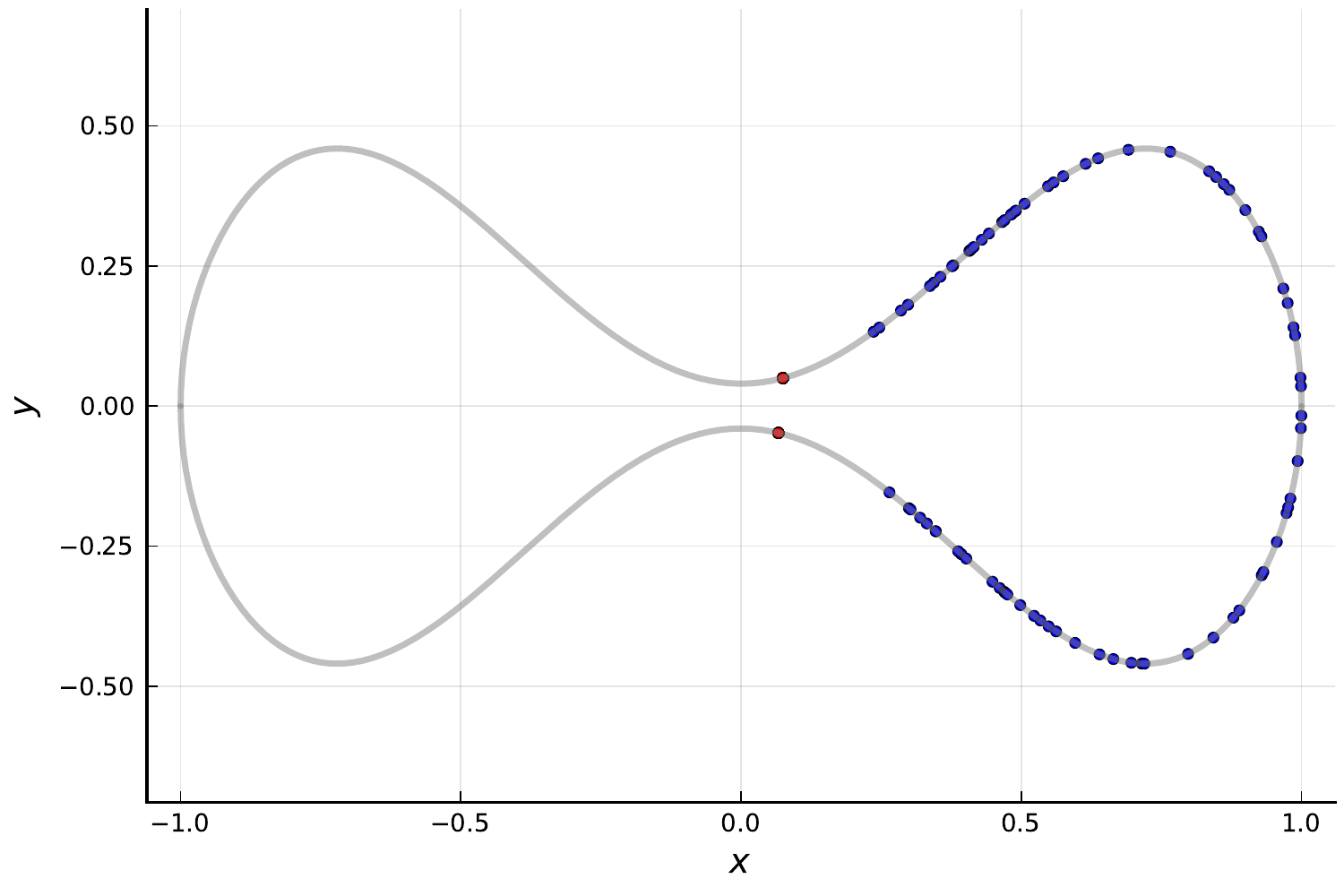}}
      \subfloat[particle trajectories.]{	\includegraphics[scale=0.52]{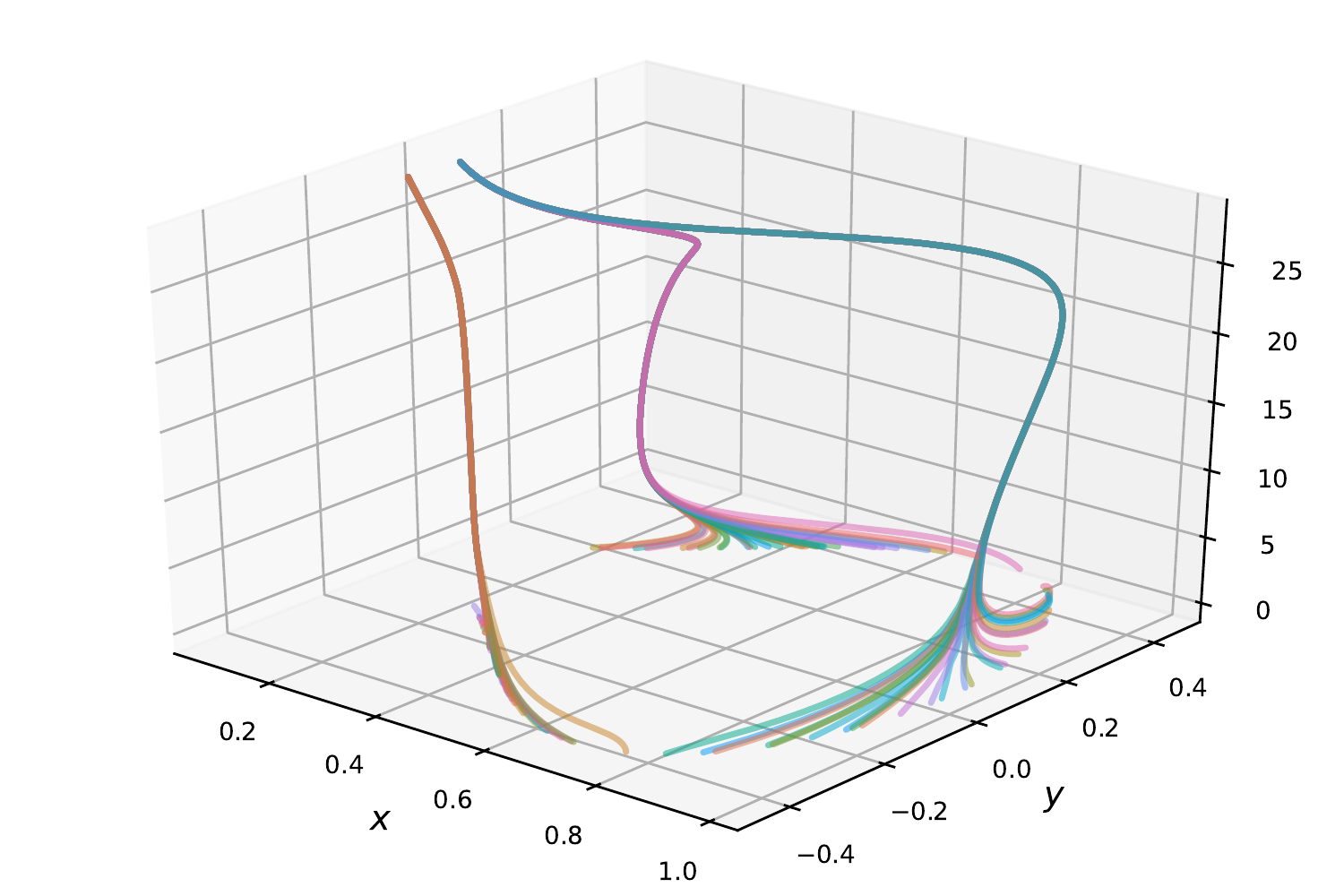}}
       \caption{Dynamics of \eqref{eq:main} approximated by \eqref{eq:projected-discrete} with domain the boundary of a bean shape for a repulsive potential.}
       \label{fig:bean_bdry}
\end{figure}

\section*{Acknowledgements}
DS is grateful to NSF for support via grant DMS 1814991.
DS and FSP are grateful to the Center for Nonlinear Analysis of CMU for its support.

\bibliographystyle{abbrv}
\bibliography{interaction_manifold}

\end{document}